\documentclass[11pt]{amsart}

\usepackage{amssymb}
\usepackage{amsthm}
\usepackage{amsmath}
\usepackage{graphicx}
\usepackage{amsaddr}
\usepackage{comment}
\usepackage[usenames,dvipsnames]{xcolor}
\usepackage[margin = 1in]{geometry}
\usepackage{enumerate}
\usepackage[toc,page]{appendix}
\usepackage{lineno}
\usepackage{color}
\usepackage[ruled,vlined,linesnumbered,noresetcount]{algorithm2e}
\usepackage{hyperref}
\usepackage{mhchem}
\usepackage{siunitx}
\usepackage{accents}

\newcommand{\blue}{}
\definecolor{mygreen}{rgb}{0.05,0.5,0.8}

 \newtheorem{thm}{Theorem}[section]
 
 \newtheorem{lem}[thm]{Lemma}
 \newtheorem{prop}[thm]{Proposition}

 \theoremstyle{definition}

 \theoremstyle{remark}
 \newtheorem{rem}{Remark}

 \numberwithin{equation}{section}

\newcommand{\cl}{\scriptscriptstyle{\text{CL}}}
\newcommand{\nlg}{\scriptscriptstyle{\text{LG}}}
\newcommand{\nsg}{\scriptscriptstyle{\text{SG}}}
\newcommand{\nsl}{\scriptscriptstyle{\text{SL}}}

\DeclareMathOperator{\Bo}{Bo}
\DeclareMathOperator{\argmin}{argmin}
\DeclareMathOperator{\proj}{Proj}
\DeclareMathOperator{\Ray}{Ray}

\newcommand{\ssr}{\xi}

\newcommand{\la}{\langle}
\newcommand{\ra}{\rangle}
\newcommand{\pt}{\partial}
\newcommand{\pcon}{\eta}

\newcommand{\eps}{\varepsilon}
\newcommand{\ud}{\,\mathrm{d}}
\newcommand{\8}{\infty}
\newcommand{\F}{\mathcal{F}}
\newcommand{\mm}{\mathcal{M}}

\providecommand{\bbs}[1]{\left(#1\right)}

\newcommand{\llra}{\Longleftrightarrow}

\begin{document}

\title[Projection method for contact line dynamics]{Projection method for  droplet dynamics on groove-textured surface with merging and splitting}

\author[Y. Gao]{Yuan Gao}
\address{Department of Mathematics, Duke University, Durham, NC}
\email{yuangao@math.duke.edu}

\author[J.-G. Liu]{Jian-Guo Liu}
\address{Department of Mathematics and Department of
  Physics, Duke University, Durham, NC}
\email{jliu@math.duke.edu}

\begin{abstract}
The geometric motion of small droplets placed on an impermeable textured substrate is mainly driven by the capillary effect, the competition among surface tensions of three phases at the moving contact lines, and the impermeable  substrate obstacle.   After introducing an infinite dimensional manifold with an admissible tangent space on the boundary of the manifold, by Onsager's principle for an obstacle problem, we derive the associated parabolic variational inequalities. These variational inequalities can be used to simulate the  contact line dynamics with unavoidable merging and splitting of droplets due to the impermeable obstacle.
 To efficiently solve the parabolic variational inequality,  
we propose an unconditional stable explicit boundary updating scheme coupled with a projection method. The explicit boundary updating efficiently decouples the computation of the motion by mean curvature of the capillary surface and the moving contact lines. Meanwhile, the projection step efficiently splits the difficulties brought by the obstacle and  the motion by mean curvature of the capillary surface. Furthermore, we prove the unconditional stability of the scheme and present an accuracy check. The convergence of the proposed scheme is also proved using a nonlinear Trotter-Kato's product formula under the pinning contact line assumption.
After incorporating
 the phase transition information at  splitting points, several challenging examples including splitting and merging of droplets are demonstrated.
\end{abstract}

\date{\today}

\maketitle
\section{Introduction}

The contact line dynamics for droplets  placed on an impermeable groove-textured substrate is a historical and challenging problem. 
The capillary effect caused by the interfacial energy dominates the dynamics of small droplets, particularly the contact lines (where three phases meet).    The dynamics of a liquid drop with capillary effect is essentially a mean curvature flow for the capillary surface (the interface between fluids inside the droplet and  gas surrounding the droplet) associated  with free boundaries (moving contact lines) on the  impermeable groove-textured substrate. The mechanics for the  contact lines dynamics is determined by Onsager's linear response relation between contact line speed $v_{\cl}$ and the unbalanced Young force $F_Y$; see \eqref{unbalance}.  The dynamic contact angle then tends to relax to the equilibrium contact angle (i.e., Young's angle $\theta_Y$) which is determined by the competitions among the surface tensions  of  three interfaces at the contact lines; see \eqref{young}.

 The dynamics become more complicated after considering  a nonlocal volume constraint, a   groove-textured substrate with constantly changed slope, a gravitational effect, and  unavoidable topological changes such as splitting and merging due to the impermeable groove-textured substrate. 
In this paper, we focus on  variational derivations,  numerical methods and    stability/convergence analysis for this challenging obstacle problem coupled with the  droplet dynamics, which consist of (i) an obstacle-type weak formulation of the motion by mean curvature for the capillary surface and (ii)  the moving contact line boundary conditions.

On the one hand, for a pure mean curvature flow model  with an obstacle but without contact line dynamics, we refer to \cite{Almeida_Chambolle_Novaga_2012, mercier2015mean} for local existence and uniqueness of a regular solution by constructing  a minimizing movement sequence. 
On the other hand, for the moving contact line boundary condition coupled with the quasi-static dynamics for the capillary surface,  there are some analysis results on  global existence and homogenization problems; see \cite{caffarelli2006homogenization, Grunewald_Kim_2011, Kim_Mellet_2014, feldman_dynamic_2014} for the	 capillary surface described by a harmonic equation and see \cite{Caffarelli_Mellet_2007, Caffarelli_Mellet_2007a, Chen_Wang_Xu_2013, Feldman_Kim_2018} for the capillary surface described by a spatial-constant mean curvature equation. 

However, the contact line dynamics for  droplets  with an unavoidable topological change including splitting and merging due to an impermeable   obstacle are lack of study in terms of both variational derivations, mathematical analysis and efficient algorithms.

First, we will  regard the full dynamics of a droplet with moving contact lines and an impermeable obstacle substrate as a curve on an infinite dimensional manifold $\mm$ with a boundary; see Section \ref{sec_configuration}. Then after introducing a free energy $\F$, an admissible tangent space $T\mm$ and a Rayleigh dissipation functional $Q$, by using Onsager's principle for an obstacle problem, we derive two parabolic variational   inequalities (PVI); see \eqref{PVI_tang} and \eqref{PVI_new}. PVI \eqref{PVI_new} is derived by minimizing a Rayleighian in a subset of the admissible tangent space, however it is computationally friendly and  we will prove that these two versions of PVIs are equivalent and have the same strong formulation; see Proposition \ref{prop_PVIs}.

Second, after including a textured substrate, for PVI \eqref{PVI_class}, we propose a numerical scheme based on an unconditionally stable explicit updating for moving contact lines and a projection method for the obstacle problem; see Section \ref{sec_scheme}. The explicit boundary updating efficiently decouples the computations for the moving contact line and the motion of the capillary surface. Meanwhile, the projection operator for the mean curvature flow of droplets with volume constraint has a closed formula characterization (see \eqref{tmtmtm} and Lemma \ref{lem_proj}), so the proposed numerical scheme based on  the projection method (splitting method) is very efficient and easy to implement.     The unconditional stability of the projection method is proved in Proposition \ref{prop_stability}, which focuses on the difficulties brought by the moving contact line, and the motion by mean curvature with an obstacle.  We also unveil the pure  PVI formulation misses the phase transition information at  splitting (merging) points where the interface between two phases becomes an emerged triple junction of three phases, so pure PVI  is not enough to correctly show the physical phenomena.  Therefore, we enforce the phase transition and contact line mechanism at emerged triple junctions in the last step of the projection method  to incorporate both the obstacle information and the phase transition information at splitting points; see Section \ref{sec_scheme}.

Next, under the pinning contact line assumption,  the full dynamics of droplets with an obstacle  and volume constraint can be formulated as a gradient flow of the sum of two convex functional in a Hilbert space. More precisely, let $K$ be a closed convex subset in a Hilbert Space $X_1$; see specific definition  in Section \ref{sec_semi} after including the volume constraint. Thus using an indicator functional  $I_K(u):=\left\{\begin{array}{cc}
0, \quad & u\in K;\\
+\8, \quad & u\notin K,
\end{array} \right.
$,   to seek solution $u(t)\in K$ of PVI \eqref{PVI_class} is equivalent to  seek the unique mild solution \cite{brezis1973operateurs} in $X_1$ generated by the sum of two maximal monotone operators $\nabla \F+\pt I_K$,
i.e.,
\begin{equation}\label{maineq}
\pt_t u(t) \in - (\nabla \F+\pt I_K)(u(t)), \quad \text{ for a.e. } t> 0,\, u(0)=u_0.
\end{equation}
Notice here we already take the advantage of the efficient local information from G\^ateaux derivatives $\nabla \F$ instead of using the subdifferential $\pt \F$.
In Lemma \ref{lem_proj}, for a given reference capillary profile $u^*$ satisfying impermeable obstacle condition $u^*\geq 0$ and the volume constraint $\int u^* \equiv V$, we give a resolvent characterization for the projection operator $\proj_K(\tilde{u}^{k+1}-u^*)$ in $X_1$
 \begin{equation}\label{tmtmtm}
\left\{
\begin{array}{cc}
u^{k+1} = \max\{ \tilde{u}^{k+1} + \lambda ,\,  0\},\\
\int_{D} u^{k+1} \ud x \ud y = V,
\end{array}
\right. \llra u^{k+1} =u^*+ (I+\tau \pt I_K)^{-1} (\tilde{u}^{k+1}-u^*).
\end{equation} 
Hence in Theorem \ref{thm1}, based on the alternate resolvent reformulation of our projection method
$$u^{k+1}= (I+\tau \pt I_K)^{-1} (I+\tau \nabla \F)^{-1} u^k = \proj_K  (I+\tau \nabla \F)^{-1} u^k,$$  
we apply a nonlinear version of Trotter-Kato's product formula \cite{Kato_Masuda_1978} to finally prove the convergence of the projection method in $X_1$.
The spirit of the projection method is same as the one for  other efficient splitting methods when solving some important physical problems such as   incompressive Navier-Stokes equation \cite{chorin1969convergence, temam1970quelques} and the Landau-Lifshitz equation \cite{weinan2001numerical}.

Finally,  in Section \ref{sec_simu}, we use a projected triple Gaussian as an initial droplet profile to check the order of convergence of the proposed projection method. Both the moving contact line, the capillary surface and the dynamic contact angle have a perfect first order accuracy. Then   several numerical simulations are conducted including the splitting of one droplet on an inclined groove-textured substrate and the merging of two droplets in a Utah teapot.

There are also many other numerical methods for simulating the geometric motion of droplets with moving contact lines or for general geometric equations with an obstacle, c.f. \cite{Leung_Zhao_2009, Morland_1982, Zhao_Chan_Merriman_Osher_1996,  Esedoglu_Tsai_Ruuth_2008, Lu_Otto_2015, xu2016reinitialization,  xu2017efficient, Wang_Wang_Xu_2019, barrett2020parametric} and the references therein. 
Particularly,  we compare with those closely related to the  geometric motion of droplets with moving contact lines and an impermeable obstacle.   The mean curvature flow with obstacles is theoretically studied in \cite{Almeida_Chambolle_Novaga_2012} in terms of a weak solution constructed by a minimizing movement (implicit time-discretization). The $L^2$ penalty method for the obstacle problem  is introduced in \cite{glowinski, Scholz_1984} and recently an advanced $L^1$ penalty method is introduced in \cite{Tran_Schaeffer_Feldman_Osher_2015, Zosso_Osting_Xia_Osher_2017}. They replace the indicator function $I_K(u)$ in the total energy by a $L^2$ (resp.  $L^1$) penalty $\mu\|(g-u)_+\|_{L^2}$ (resp. $\mu\|(w-u)_+\|_{L^1}$ )  with a large enough parameter $\mu$.  The threshold dynamics method based on characteristic functions is first used  in \cite{xu2017efficient,  Wang_Wang_Xu_2019} to simulate the contact line dynamics, which is particularly efficient and can be easily adapted to droplets with topological changes.  The authors  extended the original threshold method for mean curvature flows to  the case with a solid substrate and a free energy with multi-phase surface tensions,  in the form of obstacle problems. However, since they do not enforce the contact line mechanism \cite{deGennes_1985, ren2007boundary}, i.e., relation between the contact line speed and the unbalanced Young force $v_{\cl}=\frac{\gamma_{\nlg}}{\ssr}\left(\cos \theta_Y -\cos \theta_{\cl}\right)$, thus their computations on the moving contact line and the dynamic contact angles  are different with the present paper and only the  equilibrium Young angle $\theta_Y$ is accurately recovered. Instead, the contact line mechanism in  \cite{Jiang_Bao_Thompson_Srolovitz_2012, Wang_Jiang_Bao_Srolovitz_2015} are accurately recovered at each time step. With some  local treatments at  splitting points, the authors simulate the pinch off due to an impermeable substrate of solid drops described by the  surface diffusion with either a sharp interface dynamics or a corresponding phase field model. We point out the explicit front tracking method based on Lagrangian coordinates in \cite{Wang_Jiang_Bao_Srolovitz_2015} is convenient for simulation of the pinch-off. However, the Courant–Friedrichs–Lewy condition constraint for this explicit method is severe. 
Besides, the level set method developed in \cite{li2010augmented, xu2016reinitialization} can not be directly used and also can not deal with textured substrates which lead to PVIs for obstacle problems. For other general geometric equations including the motion by mean curvature for a two phase flow model, we refer to a review article \cite{barrett2020parametric} and references therein for a  parametric finite element  method, which is particularly useful for high dimensional problems.

The remaining part of the  paper will be organized as follows.
 In Section \ref{sec_gf}, we derive the associated parabolic variational inequalities for the contact line dynamics by Onsager's principle for an obstacle problem. In Section \ref{sec_ss}, we propose an unconditionally stable explicit boundary updating scheme coupled with the projection method for solving PVI  and  the phase transition information for merging and splitting. The stability and convergence analysis are given in Section \ref{sec_stability} and Section \ref{sec_conv}. In Section \ref{sec_simu}, we give an accuracy check for the projection method and conduct several  simulations including merging and splitting of droplets on groove-textured substrate.

\section{Parabolic variational inequalities of droplet dynamics derived by Onsager's principle with an obstacle}\label{sec_gf}
In this section, we derive the parabolic variational inequality for the contact line dynamics with an impermeable substrate as an obstacle.    We first introduce the configuration state ( the contact domain and the capillary surface) and the free energy of the  contact line dynamics.  Then we regard the configuration space as an infinite dimensional manifold with a boundary and compute the first variation of the free energy on this manifold; see Section \ref{sec_2.1} and Section \ref{sec_2.2}. Next, after introducing the admissible tangent space and a Rayleigh dissipation functional, we apply Onsager's principle for the obstacle problem to derive two associated parabolic variational inequalities (PVIs), which can be used to describe the droplet dynamics on a rough substrate with unavoidable merging and splitting; see Section \ref{sec_PVI}.       Finally, we show the equivalence  between two versions of PVI in Proposition \ref{prop_PVIs}, one of which will be used to design efficient numerical schemes for the contact line dynamics with an  impermeable obstacle.   
\subsection{Configuration states and the  free energy}\label{sec_2.1}
We study the motion of a three-dimensional droplet placed on an impermeable substrate $\{(x,y,z); z=0\}$. Let the wetting domain (a.k.a. the contact domain) be $(x,y)\in D \subset\mathbb{R}^2$ with boundary $\Gamma:=\pt D$ (physically known as the contact lines). We focus on the case that the capillary surface $S$ (the interface between the liquid and the gas) of the droplet is described by  a graph function $u(x,y)$. The droplet domain is then identified by the area $$A:=\{(x,y, z);~ (x,y)\in D, \,0<z<u(x,y),\, \, u|_{\Gamma}=0\}$$ with sharp interface $S:=\{(x,y, u(x,y)),\quad (x,y)\in D\}$. We will give a kinematic description, and a driven energy in this section.
{\blue  \subsubsection{Configuration state, geometric quantities and kinematic description for velocities  of a droplet}\label{sec_potential}
First, from the wetting domain and the capillary surface defined above,   a configuration state of a droplet  is chosen to be $(\Gamma, u)$ with $u|_\Gamma=0$.  

Second, given a configuration state $(\Gamma, u)$ with $u|_\Gamma=0$, we  clarify  the following
geometric quantities; see Fig \ref{fig_show} (a).
 The unit outer normal on the capillary surface is $n:= \frac{1}{\sqrt{1+|\nabla u|^2}}(-\nabla u, 1).$ The unit outer normal at the contact line $\Gamma$ is $n_{\cl} := \frac{-\nabla u}{|\nabla u|}\Big|_\Gamma$, which in 3D is extended as $(n_{\cl},0)$.
Define the contact angles (inside the droplet $A$) as  $\theta_{\cl}$ satisfying 
\begin{equation}\label{con-ang-n}
\sin \theta_{\cl} = n \cdot (n_{\cl}, 0) = \frac{|\nabla u|}{\sqrt{1+|\nabla u|^2}},
\end{equation}
which implies
$
\tan \theta_{\cl} = | \nabla  u|
$ at $\Gamma$.

Third, given a configuration state $(\Gamma, u)$ with $u|_\Gamma=0$, we describe two velocities of the droplet and their relations. 
(i) The motion of contact line $\Gamma$ with outer normal $n_{\cl}$  is described by the contact line speed $v_{\cl}$.
(ii) The motion of the capillary surface  with the outer normal $n$ is described by the normal speed $v_n = \frac{\pt_t u }{\sqrt{1+|\nabla u|^2}} $. Here $\pt_t u$ is the vertical velocity of the capillary surface , which is convenient to use in the graph representation.
(iii) The continuity equation $u(\Gamma(t), t)\equiv 0$ gives the relation between $\pt_t u|_\Gamma$ and $v_{\cl}$.
\begin{equation}\label{com-angle}
\begin{aligned}
\frac{\ud u(\Gamma(t),t)}{\ud t} =& \pt_t u(\Gamma(t),t) + \nabla  u(\Gamma(t),t)\cdot {\pt_t \Gamma}=\pt_t u(\Gamma(t),t) +  (\nabla  u(\Gamma(t),t)\cdot n_{\cl}) v_{\cl} \\
=&\pt_t u(\Gamma(t),t) -|\nabla  u(\Gamma(t),t)| v_{\cl}=0,
\end{aligned}
\end{equation}
where we used the fact that $n_{\cl} := \frac{-\nabla u}{|\nabla u|}\Big|_\Gamma$.

An important case is to assume the  volume preserving constraint
$
 \int_{D(t)} u(x,y,t) \ud x \ud y=V.
$
In this case, 
 by $u(x,y,t)=0$ on $\Gamma(t)$ and the Reynolds  transport theorem, we have
\begin{equation}\label{reynolds}
\int_{D(t)} \pt_t u \ud x \ud y = \frac{\ud}{\ud t} \int_{D(t)} u(x,y,t) \ud x \ud y=0,
\end{equation}
which gives an additional constraint on the vertical velocity.
}
\begin{figure}
\includegraphics[scale=0.45]{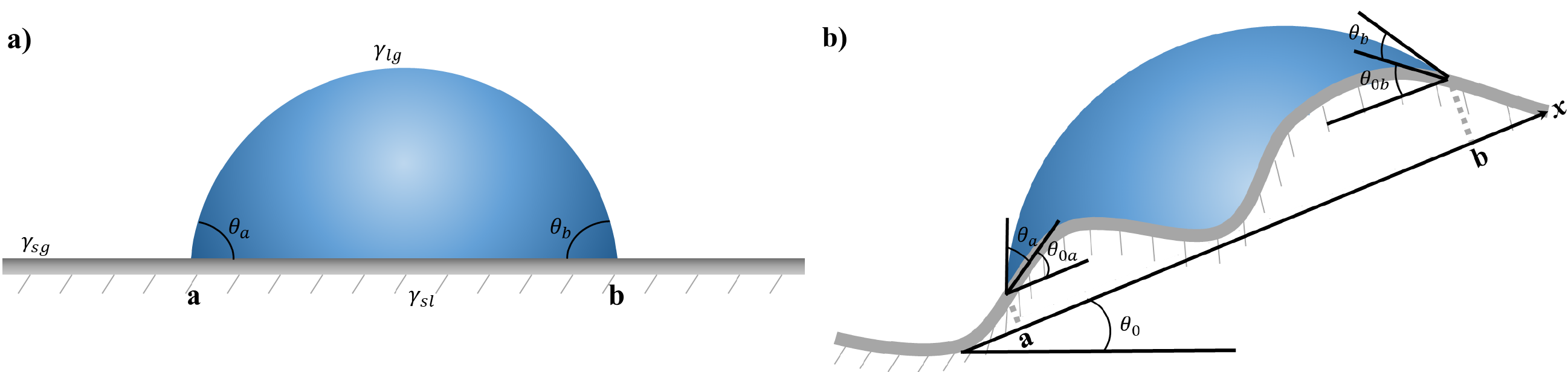} 
\caption{Droplets with contact angles $\theta_a, \theta_b$ and surface tensions $\gamma_{\nlg}, \gamma_{\nsg}, \gamma_{\nsl}$. (a) Droplet placed on $z=0$; (b) Droplet placed on an inclined groove-textured surface with effective angle $\theta_0$.}\label{fig_show}
\end{figure}

\subsubsection{Free energy for the droplet and Young's angle}\label{sec_energy} Now we clarify the  free energy of a droplet following the notations and terminologies in the classical book of De Gennes \cite{de2013capillarity}. For a droplet placed on a substrate, the surface tension contributes the leading effect to the dynamics and the equilibrium of the droplet. Especially, for  contact line $\Gamma$, where three phases of materials (gas, liquid, and solid) meet, one should consider the interactions between their surface energy. Denote  $\gamma_{\nsl}>0$ ($\gamma_{\nsg}, \gamma_{\nlg}$ resp.) as the interfacial surface energy density (a.k.a. surface tension coefficients) between solid-liquid phases (solid-gas, liquid-gas resp.). To measure the total area of the capillary surface with surface tension $\gamma_{\nlg}$ and the area of the contact domain with the relative surface tension $\gamma_{\nsl}-\gamma_{\nsg}$,  
we take the total free energy  of the droplet as the summation of the surface energy and the  gravitational energy 
\begin{equation}\label{energy}
\begin{aligned}
\F=& \gamma_{\nlg} \int_{\pt A(t) \cap \{u>0\}} \ud s + (\gamma_{\nsl}-\gamma_{\nsg}) \int_{D(t)} \ud x\ud y+ \rho g  \int_{D(t)}\frac{u^2}{2} \ud x \ud y,\\
=&    \gamma_{\nlg} \int_{D(t)} \sqrt{1+ |\nabla  u|^2} \ud x \ud y + (\gamma_{\nsl}-\gamma_{\nsg})\int_{D(t)} \ud x\ud y+ \rho g  \int_{D(t)}\frac{u^2}{2} \ud x \ud y ,\\
\end{aligned}
\end{equation}
where $\rho$ is the density of the liquid, $g$ is  the gravitational acceleration. 
Besides gravity, we neglect other forces, such as  inertia effect, viscosity stress inside the droplet, Marangoni effect, electromagnetic fields, evaporation and condensation, etc. 

With  a fixed volume $V$,  competitions between the three surface tensions  will determine uniquely the steady state of the droplet, i.e. the minimizer of $\F$.  
 Define  $\sigma$ as the  relative adhesion coefficient between the liquid and the solid
$$\sigma:= \frac{\gamma_{\nsl}-\gamma_{\nsg}}{\gamma_{\nlg}}.$$
 We remark that the spreading parameter $S:=\gamma_{\nlg}\left( \frac{\gamma_{\nsg}-\gamma_{\nsl}}{\gamma_{\nlg}} -1\right)$ could be positive in the so-called total wetting regime \cite[Section 1.2.1]{deGennes_1985}.  
 For the present contact angle dynamics setup, $S<0$ implies $\sigma>-1$.  By Young's equation \cite{young1805iii}, the equilibrium contact angle 
$\theta_Y$ is determined by the Young's angle condition
\begin{equation}\label{young}
\cos \theta_Y= \frac{\gamma_{\nsg}-\gamma_{\nsl}}{\gamma_{\nlg}} = -\sigma.
\end{equation}
We will only focus on the partially wetting (hydrophilic) case $-1<\sigma<0$, or equivalently $0<\theta_Y< \frac{\pi}{2}.$ In this case, 
adhesive forces between the liquid and the solid tend to spread the droplet across the surface and there is a vertical graph representation of the capillary surface. We refer to \cite{gao2020gradient} for more discussions on   dewetting or non-wetting droplets (i.e. $0<\sigma<1$) with a horizontal graph representation for the  quasi-static case.

\subsection{Configuration space as a manifold and the first variation of the free energy}\label{sec_2.2}
With the specific  driven energy $\F$, we compute the first variation of the free energy for any given virtual displacement. Before that, we first clarify the configuration space as a manifold and define the tangent plane at each point on the manifold.
 \subsubsection{Configuration space for the obstacle problem: a manifold with a boundary}\label{sec_configuration}
{\blue
  Here we first give a derivation by taking rough impermeable substrate as $z=0$ for simplicity.
We use an infinite dimensional manifold \cite{Klingenberg_Wilhelm} to describe the configuration space
\begin{equation}\label{mm-n}
\mm:= \{(\Gamma, u);\, \Gamma \in C^1, \, u\in C^2(D)\cap L^2(D),~u\geq 0 \text{ on }D, ~ u|_\Gamma=0\}.
\end{equation}
The dynamics of the droplet is represented by a trajectory on this manifold.
Consider a trajectory $\pcon(t) \in \mm$ starting from initial state 
$\pcon(0)=\{\Gamma(0), u(x,y,0)\}\in\mm$,
\begin{equation}
\pcon(t)=\{\Gamma(t), u(x,y,t)\}\in \mm, \quad t\in[0,T]. 
\end{equation}

\subsubsection{ Obstacle constraint in the tangent plane: a convex closed subset }
Given a configuration state $\eta=\{\Gamma, u\}$, now we use the vertical velocity $v=\pt_t u$ and the contact line speed $v_{\cl}$ to describe the tangent plane $T_{\pcon}\mm$. 
To maintain the continuity at the contact line, these two velocities in the tangent plane satisfy the linear  restriction \eqref{com-angle}.

Since the geometric motion has an obstacle  condition $u\geq 0$,  define the coincidence set as
\begin{equation}
B_\eta:= \{(x,y)\in D;  u(x,y)=0 \},
\end{equation}
which is a closed subset of $D$.
Manifold $\mm$ has a boundary, i.e. $\{\eta\in\mm; B_\eta \neq \emptyset\}$.
 If $B_\eta \neq \emptyset$, then $\eta$ is on the boundary of the manifold $\mm$. In other words, the capillary surface $u$ touches the obstacle on the coincidence set, which will lead to a PVI as described below in Section \ref{sec_PVI}. 

Given $\eta=\{\Gamma, u\}$, define the following weighted $L^2$ space as the ambient space
\begin{equation}
X_{\eta}:=L^2(\Gamma \times D; \, \ud s \times  \frac{1}{\sqrt{1+|\nabla u|^2}} \ud x \ud y)
\end{equation}
with the weighted inner product
\begin{equation}
\la q_1, q_2 \ra_{X_\eta}:= \ssr \int_{\Gamma(t)}v_{\cl 1} v_{\cl 2} \ud s+ \zeta\int_{D(t)} v_1   v_2 \frac{\ud x \ud y}{\sqrt{1+ |\nabla u|^2}}
\end{equation}
for any $q_1=(v_{\cl 1}, v_1),\, q_2=(v_{\cl 2},v_2)\in X_\eta.$ Here the constants $\ssr$ and $\zeta$ are indeed the friction coefficients which will be explained later.

Then the tangent plane 
\begin{equation}\label{tm-n}
T_{\pcon} \mm:=\{(v_{\cl}, v)\in C^1(\Gamma\times D); ~v - |\nabla  u(\Gamma(t),t)| v_{\cl} =0,~~~v\geq 0 \text{ on }   B_\eta\} \subset X_\eta
\end{equation}
is a closed convex cone and is  embedded in the ambient space $X_\eta$ with the same inner product as $X_\eta$.
We remark the last inequality for $T_\eta\mm$ in \eqref{tm-n} becomes effective when $\eta$ sits on the boundary of the manifold $\mm$, i.e. $B_\eta \neq \emptyset$. If $B_\eta =\emptyset$, then the tangent plane $T_\eta \mm$ is a linear subspace of $X_\eta$.
}

\subsubsection{First variation of the free energy}
{
\blue
For an arbitrary trajectory $\tilde{\pcon}(s)=\{\tilde{\Gamma}(s), \tilde{u}(x,y,s)\}$ (a.k.a.  virtual displacement) starting from $\tilde{\pcon}(t)=\eta(t)$ at the tangent direction
\begin{equation}\label{direction}
 \tilde{\pcon}'(t)=\{\pt_t\tilde{\Gamma}(t), \pt_t \tilde{u}\} = \{\tilde{v}_{\cl} n_{\cl}, \tilde{v}\},  \quad \{\tilde{v}_{\cl},  \tilde{v}\} \in T_{\pcon(t)} \mm,
\end{equation}
we know
\begin{equation}\label{com-con-n}
\tilde{v}|_\Gamma=  |\nabla  u(\Gamma(t),t)| \tilde{v}_{\cl}.
\end{equation}
To ensure the volume preserving condition $\int_{D(t)} u \ud x\ud y = V, \, t\in[0,T]$, we calculate the first variation of extended free energy $\F(\eta,\lambda)$ on manifold $\mm\times \mathbb{R}$ for $\eta(t)\in \mm$ and a Lagrange multiplier $\lambda(t)$
\begin{equation}
\F(\eta(t), \lambda(t))= \F(\eta(t))- \lambda(t)(\int_{D(t)} u(t) \ud x \ud y -V).
\end{equation}
Then we have
\begin{equation}\label{energydissipation}
\begin{aligned}
& \frac{\ud }{\ud s}\big|_{s=t^+}  \F(\tilde{\pcon}(s), \tilde{\lambda}(s)) \\
=& \frac{\ud }{\ud s}\big|_{s=t^+} \F(\tilde{\eta}(t))- \frac{\ud }{\ud s}\big|_{s=t^+} \tilde{\lambda}(s)(\int_{\tilde{ D}(s)} \tilde{u}(s) \ud x \ud y -V)\\
=&  \frac{\ud }{\ud s}\big|_{s=t^+} \F(\tilde{\pcon}(s)) - \lambda(t)\int_{D(t)}\tilde{v}(t) \ud x \ud y - \tilde{\lambda}'(t)(\int_{D(t)} u \ud x \ud y -V),
 \end{aligned}
\end{equation} 
where we used \eqref{direction} and the  Reynolds  transport  \eqref{reynolds}  in the last equality.

Now we calculate the first variation for a generic $\F(\tilde{\pcon}(s))$ with the energy density function $G(u, \nabla u)$. From \eqref{com-con-n} and the Reynolds transport theorem, we have 
\begin{align}
&\frac{\ud }{\ud s}\big|_{s=t^+} \int_{\tilde{D}(t)} G( \tilde{u}(x,y,s), \nabla \tilde{u}(x,y,s)) \ud x \ud y \nonumber\\
=& \int_{\Gamma(t)} G|_\Gamma \tilde{v}_{\cl} \ud s + \int_{D(t)} \pt_u G \tilde{v} + \pt_{ \nabla u} G \cdot \nabla \tilde{v} \ud x \ud y \nonumber\\
=&  \int_{\Gamma(t)} G|_\Gamma \tilde{v}_{\cl} \ud s + \int_{D(t)} (\pt_u G - \nabla \cdot  (\pt_{ \nabla u} G) ) \tilde{v} \ud x \ud y +  \int_{\Gamma(t)}\tilde{v} (n_{\cl} \cdot \pt_{\nabla u} G)\ud s \nonumber\\
=& \int_{\Gamma}[G+|\nabla  u |(n_{\cl} \cdot \pt_{\nabla u} G) ]\big|_\Gamma \tilde{v}_{\cl} \ud s +   \int_{D(t)} (\pt_u G - \nabla \cdot  (\pt_{ \nabla u} G) ) \tilde{v} \ud x \ud y. \label{var}
\end{align}
Regard the contact line $\Gamma$ and the capillary surface $S$ as an open system.  Denote the two forces exerted by the droplet (the open system $\Gamma$ and $S$) to the environment as
\begin{equation}\label{unbalance}
F:=-\pt_u G + \nabla \cdot  (\pt_{ \nabla u} G) , \quad F_Y := -G-|\nabla  u |(n_{\cl} \cdot \pt_{\nabla u} G).
\end{equation}
Then  free energy dissipation can be expressed as the virtual work per unit time  done by the two virtual forces to the environment.
For the case that $G = \gamma_{\nlg} \sqrt{1+|\nabla u|^2} + (\gamma_{\nsl}-\gamma_{\nsg}) + \rho g \frac{u^2}{2}-\lambda u$, we know
\begin{equation}
F = -\rho g u + \gamma_{\nlg}\nabla \cdot \bbs{\frac{\nabla u}{\sqrt{1+|\nabla u|^2}}}+ \lambda = -\rho g u  +\gamma_{\nlg}  H + \lambda,  \quad F_Y =  \gamma_{\nlg} (\cos \theta_{\cl} - \cos \theta_Y),
\end{equation}
where $H = \nabla \cdot \bbs{\frac{\nabla u}{\sqrt{1+|\nabla u|^2}}} $ is the mean curvature.

With these two unbalanced forces, we know
\begin{equation}
\begin{aligned}
& \frac{\ud }{\ud s}\big|_{s=t^+}  \F(\tilde{\pcon}(s), \tilde{\lambda}(s)) \\
=&  - \int_\Gamma F_Y \tilde{v}_{\cl} \ud s - \int_{D(t)} F \tilde{v} \ud x \ud y   - \tilde{\lambda}'(t)(\int_{D(t)} u \ud x \ud y -V)
 \end{aligned}
\end{equation}
is a linear functional in terms of $(\tilde{v}_{\cl}, \tilde{v}, \tilde{\lambda}')$.
Then after imposing the volume constraint $\int_{D(t)} u \ud x \ud y =V$,   the energy dissipation is given by
\begin{equation}\label{virtual_work}
\begin{aligned}
\frac{\ud}{\ud s}\Big|_{s= t^+} \F(\tilde{\eta}(s), \tilde{\lambda}(s)) = - \int_\Gamma F_Y \tilde{v}_{\cl} \ud s - \int_{D(t)} F \tilde{v} \ud x \ud y =: \dot{\F}((\tilde{v}_{\cl}, \tilde{v}); \eta).
\end{aligned}
\end{equation}
We will see below, since the Rayleigh dissipation functional only has quadratic dissipation in terms of  $\tilde{v}_{\cl}$ and $\tilde{v}$, in order to ensure the Rayleighian is bounded below,  we need enforce the volume constraint $\int_{D(t)} u \ud x \ud y =V$. Otherwise, a relaxation model by introducing a dissipation in terms of $\lambda$ shall be used.
}

\subsection{ Onsager's principle and PVI for an obstacle problem}\label{sec_PVI}
With the specific  driven energy $\F$ on manifold $\mm$ and its first variation, we now start to derive the droplets dynamics using Onsager's principle with the obstacle as described below.

 \subsubsection{Friction damping for the motion of the droplet and the Rayleigh dissipation function}\label{sec_dissi}
From \eqref{virtual_work},  the droplet experiences unbalanced forces $(-F_Y, -F)$ exerted by the environment. These two forces can be modeled as friction forces done by the environment due to the motion of the droplet.  First,  the contact line friction force density is given by $-F_Y = -\ssr v_{\cl}$, where $\ssr $ is the friction damping coefficient per unit length for the contact line with the units of  mass/(length $\cdot$ time). Second,  the   friction force density on the capillary surface is given by $-F=-\zeta v_n $, where $\zeta$ is the friction damping coefficient per unit area for the capillary surface with the units of mass/(area $\cdot$ time). 
These are the  simplest linear response relations between the unbalanced forces $(F_Y, F)$ and the velocities $(v_{\cl}, v_n)$. 
If one also considers the  viscosity dissipation due to the fluids surrounding the capillary surface, we refer to \cite{gaoliu21} for a nonlocal linear response relation.

    Then we introduce the Rayleigh dissipation functional  (in the unit of work per unit time)   given by \cite{goldstein2002classical}
\begin{equation}\label{RQ}
Q=\frac{\ssr }{2} \int_{\Gamma (t)} |v_{\cl}|^2 \ud s +  \frac{\zeta}{2} \int_{S(t)} |v_n|^2   \ud s.
\end{equation} 
{\blue With the geometric configurations, contact line $\Gamma(t)$ and  capillary surface $u(x,y,t)$, the variation of free energy \eqref{energy} and Rayleigh's dissipation functional \eqref{RQ}, in the next section, we give detailed derivations of the governing equations using Onsager's principle with an obstacle.
}

\subsubsection{Euler-Lagrange equations derived by Onsager's principle}
{\blue 
Recast the Rayleigh dissipation functional as functional of $(v_{\cl}, v)$
\begin{equation}\label{QQ22}
Q((v_{\cl}, v); \eta)=\frac{\ssr }{2} \int_{\Gamma (t)} |v_{\cl}|^2 \ud s +  \frac{\zeta}{2} \int_{D(t)} |v|^2 \frac{1}{\sqrt{1+|\nabla u|^2}}   \ud x \ud y.
\end{equation}
Define the Rayleighian as 
\begin{equation}\label{eq:Ray}
\Ray((v_{\cl}, v); \eta):=  Q((v_{\cl}, v); \eta) + \dot{\F} (({v}_{\cl}, {v}); \eta).
\end{equation}
Then minimizing Rayleighian w.r.t $(v_{\cl}, v) \in T_\eta M$ 
\begin{equation}
(v_{\cl}, \pt_t u) = \argmin_{(v_{\cl}, v)\in T_\eta M} \Ray((v_{\cl}, v); \eta)
\end{equation}
yields the Euler-Lagrange equations. Indeed, notice the minimization $(v_{\cl}, \pt_t u)$ satisfies 
\begin{equation}
\Ray((v_{\cl} + \eps(\tilde{v}_{\cl} - v_{\cl}), \pt_t u + \eps (\tilde{v} - \pt_t u))) \geq \Ray (v_{\cl}, \pt_t u), \quad \text{ for any } ((\tilde{v}_{\cl}, \tilde{v}))\in T_\eta \mm.
\end{equation} 
Thus taking $\eps \to 0^+$ concludes the following parabolic variational inequality (PVI):  
\begin{equation}\label{PVI_tang}
\begin{aligned}
& \text{ find } \eta(t)=(\Gamma(t), u(t)) \in \mm, \quad (v_{\cl}, \pt_t u)\in T_\eta \mm \,\,  \text{ such that }\\
&\ssr  v_{\cl} = F_Y \quad \text{ on }\Gamma, \\
& \zeta\int_{D(t)} \frac{\pt_t u}{\sqrt{1+|\nabla u|^2}} (\tilde{ v}- \pt_t u) \ud x \ud y \geq \int_{D(t)} F (\tilde{v}-\pt_t u ) \ud x \ud y, \quad  \forall \tilde{v} \in Y_\eta,\\
& \int_{D(t)} u \ud x\ud y =V,
\end{aligned}
\end{equation}
where $Y_\eta:= \{v\in L^2(D); v\geq 0  \text{ on } B_\eta\}.$

By taking $\tilde{v} = 0$ or $\tilde{v} = 2\pt_t u$, we conclude the energy dissipation relation
\begin{equation}\label{O_id}
\frac{\ud }{\ud t}\F(\eta) = -2 Q((v_{\cl},v); \eta).
\end{equation}

\subsubsection{Equivalent PVI  derived by Onsager's principle in a subset}
Define a convex subset $K:=\{ v\in C_c^1(D),    v\geq 0\}$. We first minimize the Rayleighian defined in \eqref{eq:Ray} with any $\tilde{v}_{\cl}\in C(\Gamma)$, and the associated $\tilde{v}$. This gives the same equality $\ssr  v_{\cl} = F_Y$ for the moving contact line. 

Next, we      minimize the Rayleighian defined in \eqref{eq:Ray}  in a subset $(0, \tilde{v}) \subset T_\eta \mm$ for any $\tilde{v}\in K.$ This, together with the equality for the moving contact line, gives a new PVI
\begin{equation}\label{PVI_new}
\begin{aligned}
& \text{ find } \eta(t)=(\Gamma(t), u(t)) \in \mm, \quad (v_{\cl}, \pt_t u)\in T_\eta \mm \,\,  \text{ such that }\\
&\ssr  v_{\cl} = F_Y \quad \text{ on }\Gamma, \\
& \zeta\int_{D(t)} \frac{\pt_t u}{\sqrt{1+|\nabla u|^2}} (\tilde{ v}-  u) \ud x \ud y \geq \int_{D(t)} F (\tilde{v}- u ) \ud x \ud y, \quad  \forall  \tilde{v}\in K,\\
& \int_{D(t)} u \ud x\ud y =V.
\end{aligned}
\end{equation}

Indeed, the derivation relies on the fact that  $u+\eps(\tilde{v}-u)= (1-\eps) u + \eps \tilde{v}$ is a curve on $\mm$ with parameter $0<\eps<1$. Thus we know $(0, \tilde{v}-u)\in T_\eta \mm$ and $(0, \pt_t u + \eps (\tilde{v}-u))\in T_\eta \mm$ since $T_\eta \mm$ is a convex cone. So \eqref{PVI_new} is concluded after taking minimization for any   $(0, \pt_t u + \eps (\tilde{v}-u))\in T_\eta \mm$
\begin{equation}
\Ray((v_{\cl} , \pt_t u + \eps (\tilde{v} -  u))) \geq \Ray (v_{\cl}, \pt_t u)
\end{equation} 
and letting $\eps \to 0^+$.

By taking $\tilde{v}(\cdot) = h(\cdot, t) \pm \frac{s-t}{2} \pt_t h(\cdot, t)$ and $s\to t^+,$  we conclude the same energy dissipation relation \eqref{O_id}.

\subsubsection{Strong form of PVI}
Using the coincidence set $B_\eta$, by taking different $\tilde{v}$, the second equation in \eqref{PVI_tang} or in \eqref{PVI_new} can be recast as the same strong form. 
We have the following proposition on the equivalence of  \eqref{PVI_tang} and \eqref{PVI_new}.
\begin{prop}\label{prop_PVIs}
(i) Let $(\Gamma(t), u(t))$ be a smooth solution to the PVI \eqref{PVI_tang}. Then $(\Gamma(t), u(t))\in \mm$ satisfies the Euler-Lagrange equation in the strong form
\begin{equation}\label{PVI_strong}
\begin{aligned}
&  (v_{\cl}, \pt_t u)\in T_\eta \mm \,\,  \text{ such that }\\
&\ssr  v_{\cl} = F_Y \quad \text{ on }\Gamma, \\
&\left\{
\begin{array}{cc}
\zeta\frac{\pt_t u}{\sqrt{1+|\nabla u|^2}}  = F \quad & \text{ for } (x,y)\notin B_\eta, \text{ or } F\geq 0 \\
\pt_t u = 0 & \text{ for } (x,y)\in B_\eta \text{ and } F<0.
\end{array}
\right. \quad \text{in } D,\\
& \int_{D(t)} u \ud x\ud y =V.
\end{aligned}
\end{equation}
(ii) Let $(\Gamma(t), u(t))$ be a smooth solution to the PVI \eqref{PVI_new}. Then $(\Gamma(t), u(t))\in \mm$ satisfies the same Euler-Lagrange equation \eqref{PVI_strong}.
\end{prop}

The proof of this proposition is standard and can be derived by considering different cases that appear in \eqref{PVI_strong}.
From this proposition, we know  two PVIs are equivalent and our projection method will rely on PVI \eqref{PVI_new}.

In short words, we will alternately  conduct the following two steps: \\
Step (i), solve the equality governing equations
\begin{equation}\label{MMC}
\begin{aligned}
\ssr  v_{\cl} = \gamma_{\nlg} (\cos \theta_{\cl} - \cos \theta_Y) \quad  \text{ on }\Gamma(t), \\
\zeta\frac{\pt_t u}{\sqrt{1+|\nabla u|^2}}  =  -\rho g u  +\gamma_{\nlg}  H + \lambda \quad  \text{ in }D(t),\\
u=0 \quad \text{ on }\Gamma,\\
\int_{D(t)} u \ud x \ud y =V
\end{aligned}
\end{equation}
with initial data $\eta(0)=\{\Gamma(0), u(x,y,0)\}$ and initial volume $V$. 
Then we do
Step (ii), the projection  to the manifold $\mm$.
We will discuss the schemes,  unconditional stability and convergence analysis  in details in Section \ref{sec_ss}.

}

\section{Numerical schemes, stability and convergence analysis}\label{sec_ss}
In this section, we first propose a numerical scheme for droplets dynamics with merging and splitting, which are  extensions of the 1st/2nd order schemes developed in \cite{gao2020gradient} for a single droplet without topological changes. To incorporate the splitting due to an impermeable obstacle, we need to solve the PVI \eqref{PVI_new} instead of PDEs. Inspired by a nonlinear version of Trotter-Kato's product formula, a projection method, which efficiently splits the PDE solver and the obstacle constraint,  will be adapted.
Then in Section \ref{sec_stability},  we prove the unconditional  stability of the projection method for the moving contact line coupled with the motion by mean curvature with an obstacle. In the stability analysis in Proposition \ref{prop_stability}, we focus on the key difficulties due to the moving domain and the obstacle, and consider    a droplet  placed on a horizontal plane without gravity and volume constraint. Finally, in Section \ref{sec_conv} we include the gravity and volume constraint and give the convergence analysis of the projection method  for droplets with a pinning  contact line; see Theorem \ref{thm1}.
\subsection{Numerical schemes based on explicit boundary moving and the projection method}\label{sec_scheme}

In this subsection, we present a numerical scheme for PVI \eqref{PVI_new} describing the droplet dynamics with merging and splitting.
First, we  further split the  equality solver for \eqref{MMC}  into two steps: (i) explicit boundary updates and (ii) semi-implicit capillary surface updates.
The unconditional stability for the explicit 1D boundary updates is proved in \cite{gao2020gradient}, which efficiently decouples the computations of the boundary evolution and the capillary surface updates. The semi-implicit  capillary surface updates without obstacles but with volume constraint can be convert to a standard elliptic solver at each step. Next, to enforce the impermeable obstacle,  (iii) we project the capillary surface to the manifold $\mm$. This step  has explicit formula so also keeps the efficiency. Finally, to incorporate the phase transition information at  splitting points,  (iv) we detect  splitting points after a threshold and add new contact line updates after that; see detailed explanations for the phase transition at emerged contact lines in Section \ref{sec_PVI_2D},

\subsubsection{PVI for 2D  droplet placed on a groove-textured  and inclined surface}\label{sec_PVI_2D} For simplicity of the representation, we only describe  numerical schemes for 2D droplets. Hence  we  first use the PVI obtained in \eqref{PVI_new} to derive  the governing PVI for  a 2D droplet placed on a groove-textured  and inclined surface and explain the phase transition happening at emerged contact lines after splitting.
   
Given a groove-textured  impermeable surface described by a graph function $w(x)$, a droplet is then described by
$
A:= \{(x,y); a\leq x\leq b , w(x)\leq y\leq u(x)+w(x) \}. 
$
Following the convention, we use the Cartesian coordinate system built on an inclined plane with effective inclined angle $\theta_0$ such that $-\frac{\pi}{2}<\theta_0<\frac{\pi}{2}$ and  $(\tan \theta_0) x$ is the new $x$-axis; see Figure \ref{fig_show}. Denote the height function as
 $$h(x,t):= u(x,t)+w(x).$$
To be consistent with  height function $u$ in the last section, we choose the configuration states of this droplet as the relative height function (capillary surface) $u(x,t)\geq 0$ and  partially wetting domain $a(t)\leq x \leq b(t)$ with free boundaries $a(t), b(t).$  Consider the manifold
\begin{equation}
\mm:= \{a, b, u(x);~u(x)\geq 0,\, u(x)\in C_0^2(a,b)\cap L^2(a,b)\}.
\end{equation}
Consider the energy functional associated with the groove-textured surface
\begin{equation}\label{energy-r}
\begin{aligned}
\F(\pcon(t))= &\gamma_{\nlg} \int_{a(t)}^{b(t)} \sqrt{1+ (\pt_x (u+w))^2} \ud x + (\gamma_{\nsl}-\gamma_{\nsg}) \int_{a(t)}^{b(t)} \sqrt{1+ (\pt_x w)^2} \ud x\\
&+ \rho g  \int_{a(t)}^{b(t)}\int_{h_0}^{u+h_0}(y\cos \theta_0+x\sin\theta_0)\ud y \ud x, 
\end{aligned}
\end{equation}
where $\rho$ is the density of the liquid, $g$ is  the gravitational acceleration.
Then we have
\begin{equation}\label{tm_G}
\begin{aligned}
G(u,x)=\sqrt{1+ (\pt_x (u+w))^2}+ \sigma \sqrt{1+ (\pt_x w)^2}+\frac{\rho g}{\gamma_{\nlg}} (\frac{u^2\cos\theta_0}{2}+ \cos\theta_0 u w+ x \sin\theta_0 u)\\
\pt_u G = \frac{\rho g}{\gamma_{\nlg}} \left((u+w)\cos\theta_0+x \sin\theta_0\right), \quad \pt_{u_x} G = \frac{\pt_x(u+w)}{\sqrt{1+ (\pt_x(u+w))^2}}.
\end{aligned}
\end{equation}


\begin{rem}\label{rem1}
Let the density of gas outside the droplet be $\rho_0=0$. We denote the capillary coefficient as $\varsigma:= \frac{\rho g}{\gamma_{\nlg}}>0$ and the capillary length as $L_c:=\frac{1}{\sqrt{\varsigma}}.$  For a droplet with volume $V$, its  equivalent  length (characteristic length)  $L$ is defined as $V=\frac{4\pi}{3}L^3$ in 3D and $V=\pi L^2$ in 2D.  The Bond number 
$\Bo:=(\frac{L}{L_c})^2=\varsigma L^2$ shall be small enough to observe the capillary effect \cite{de2013capillarity}. 
In the inclined case, for a droplet with volume $V$ in 2D, the effective Bond number is
$
\Bo:=\left(\frac{L}{L_c}\right)^2 \cos \theta_0=\varsigma L^2 \cos \theta_0.
$ After dimensionless argument, we use  new dimensionless quantities $\beta$, $\kappa$, $V$ in the governing equation below.
\end{rem}

Recall $K=\{ v\in C_c^1(D),    v\geq 0\}$.  Then by the same derivations as  \eqref{PVI_new}, we have
 the governing PVI for the 2D droplet
\begin{equation}\label{wet-eq-r}
\begin{aligned}
\int_{a(t)}^{b(t)}\bbs{\beta \frac{\pt_t h(x,t)}{\sqrt{1+ (\pt_xh)^2}}- \frac{\pt }{\pt x}\left(  \frac{\pt_xh}{\sqrt{1+ (\pt_xh)^2}}\right)+\kappa (h\cos\theta_0+x \sin\theta_0)-\lambda(t)} (\tilde{v}-u ) \ud x\geq 0,\\
\text{ for any }\tilde{v}(x)\in K,\\
h=u+w, \quad u\geq 0 \quad \text{ in } D(t),\\
 u(a(t), t)=u(b(t), t)=0,\\
a'(t)= \sigma  \sqrt{1+ (\pt_x w)^2}+\frac{1+ \pt_x h  \pt_x w}{\sqrt{1+(\pt_xh)^2}}\Big|_{x=a}=\frac{1}{\cos \theta_{0a}}( \cos \theta_a-\cos \theta_Y),\\
b'(t)=- \sigma  \sqrt{1+ (\pt_x w)^2}-\frac{1+ \pt_x h \pt_x w}{\sqrt{1+(\pt_xh)^2}}\Big|_{x=b}=-\frac{1}{\cos \theta_{0b}}( \cos \theta_b-\cos \theta_Y),\\
\int_{a(t)}^{b(t)} u(x,t) \ud x = V,
\end{aligned}
\end{equation}
where $\theta_a$, $\theta_b$ are two contact angles at $a(t)$, $b(t)$  and  $\pt_x w|_a = \tan \theta_{0a}$, $\pt_x h|_a= \tan(\theta_{0a}+ \theta_a)$ and $\pt_x w|_b =- \tan \theta_{0b}$ and $\pt_x h|_b= -\tan(\theta_{0b}+ \theta_b)$; see Fig \ref{fig_show}.  It is easy to check the steady state $a'(t)=b'(t)=0$ recovers  Young's angle condition.

The parabolic variational inequality (PVI) above is able to describe the merging and splitting of several drops. However, whenever topological changes happen, \eqref{PVI_new} can not describe the correct  phase transition at the splitting/merging points. For instance, when one droplet splits into two droplets,  physically, at the splitting domain $D_0:=\{(x,y)\in  \text{int} D;u(x,t)=0\}$, the interface between gas and liquid becomes the interface between gas and solid, therefore   new contact lines with competitions from three phases appear.  Instead, the dynamics governing by PVI \eqref{PVI_new} does not contain these phase transition information but only leads to a nonphysical motion at the splitting domain $D_0$, i.e. droplet is allowed to move along the boundary $D_0$.
We propose the following natural method to incorporate the phase transition information into dynamics  after splitting. (I)  We first detect when and where the phase transition happens by recording the new generated contact lines. (II) Then  surface energies from three phases take over the  dynamics posterior to splitting.
That is to say, the generated two droplets have the same governing equation with \eqref{PVI_new} respectively and the volume of each droplet is preserved over time; see Step 5 in the algorithm below.

\subsubsection{First order numerical  scheme: explicit boundary updating with the projection method }\label{sec_1st}
~~

Step 1.  Explicit boundary updates. Compute the one-side approximated derivative of $h^n$ at $b^n$ and $a^n$, denoted as $(\pt_x h^n)_N$ and $(\pt_x h^n)_0$. Then by  the  dynamic boundary condition in \eqref{wet-eq-r}, we update $a^{n+1}, b^{n+1}$ using 
\begin{equation}
\begin{aligned}\label{end_a}
\frac{a^{n+1}-a^n}{\Delta t}&= \sigma \sqrt{1+ (\pt_x w)_0^2} +\frac{1+ (\pt_x h^n)_0(\pt_x w)_0 }{\sqrt{1+ (\pt_x h^n)_0^2}} , \quad\\
\frac{b^{n+1}-b^n}{\Delta t}&= -\sigma \sqrt{1+ (\pt_x w)_N^2} -\frac{1+ (\pt_x h^n)_N(\pt_x w)_N }{\sqrt{1+ (\pt_x h^n)_N^2}}.
\end{aligned}
\end{equation}

Step 2. Rescale $h^n$ from $[a^n, b^n]$ to $[a^{n+1}, b^{n+1}]$ with $O(\Delta t ^2)$ accuracy using a  arbitrary Lagrangian-Eulerian discretization. 
For $x^{n+1}\in[a^{n+1}, b^{n+1}]$, denote the map from moving grids at $t^{n+1}$ to $t^n$ as
\begin{equation}
x^n:= a^n + \frac{b^n-a^n}{b^{n+1}-a^{n+1}}(x^{n+1}-a^{n+1})\in[a^n, b^n].
\end{equation}
Define the rescaled solution for $h^n$ as
\begin{equation}\label{inter-u-0}
    h^{n*}(x^{n+1}):= h^n(x^n)+ \pt_x h^n (x^n)(x^{n+1}-x^n).
\end{equation}
 It is easy to verify by using the Taylor expansion
$
    h^{n*}(x^{n+1}) = h^n(x^{n+1}) + O(|x^n-x^{n+1}|^2);
$
see \cite[Appendix B]{gao2020gradient}

Step 3. Capillary surface updates without impermeable obstacle constraint, but with volume preserving constraint. Update $\tilde{ h}^{n+1}$ and $\lambda^{n+1}$ semi-implicitly.
\begin{equation}\label{tm313}
\begin{aligned}
\frac{\beta}{\sqrt{1+ (\pt_x h^{n*})^2}} \frac{\tilde{ h}^{n+1}-h^{n*}}{\Delta t}= \frac{\pt }{\pt x}\left( \frac{\pt_x \tilde{ h}^{n+1}}{\sqrt{1+ (\pt_x h^{n*})^2}} \right)-\kappa (h^{n+1}\cos \theta_0 + x^{n+1}\sin \theta_0)+\lambda^{n+1}, &\\
\tilde{ h}^{n+1}(a^{n+1})=w(a^{n+1}), \quad \tilde{ h}^{n+1}(b^{n+1})=w(b^{n+1})&,\\
\int_{a^{n+1}}^{b^{n+1}} \tilde{u}^{n+1}(x^{n+1}) \ud x^{n+1} = V,&
\end{aligned}
\end{equation}
where the independent variable  is $x^{n+1}\in(a^{n+1},b^{n+1})$.

Step 4. Enforce impermeable obstacle condition by the  projection.
Find $h^{n+1}$ and  $\lambda$ satisfying
\begin{equation}\label{bisect}
\left\{
\begin{array}{cc}
h^{n+1} = \max\{ \tilde{h}^{n+1} + \lambda , w\},\\
\int_{D} h^{n+1} \ud x \ud y = V + \int_{D} w \ud x \ud y.
\end{array}
\right.
\end{equation}
This is indeed project $\tilde{h}$ to the manifold $\mm$ with the volume constraint $V$ of the droplet; see  Lemma \ref{lem_proj}. \eqref{bisect} can be implemented using a bisection search for $\lambda$.

Step 5. Phase transition and emerge triple points. Let $\eps>0$ be a threshold parameter. If the length of splitting domain $D_0>\eps$, then record two new endpoints $c^{n+1}, d^{n+1}$ (emerged triple points).  Regard the current profile $h^{n+1}$ on $(a^{n+1}, c^{n+1})$ and $(d^{n+1}, b^{n+1})$ as two independent droplets and enforce the moving contact line  boundary conditions at these two emerged triple points $c^{n+1}, d^{n+1}$. The total volume of these two droplet remains same.

 \textit{First order scheme for merging:} The numerical scheme for the dynamics of two independent droplets with endpoints $a_1^{n}, b_1^n$ ($a_2^n, b^n_2$ resp.) are same as Step 1-3. To detect the merging of two independent droplets, at each time stepping $t^n$, one also need a threshold parameter $\eps>0$ such that we treat two droplets as one big droplet if $|a_2^n-b_1^n|<\eps$.
 
 The projection method for droplets dynamics above also works for  second order scheme, which replaces Step 1-3 by midpoint schemes. We omit details and refer to \cite{gao2020gradient}.
 
 \begin{rem}
 In Step 5, the additional moving contact line boundary condition at the emerged triple points after splitting is just a  numerical algorithm to realize the phase transitions from two phases to three phases. To model this  procedure in a variational formulation is still an open question. In the following stability analysis,  convergence analysis and accuracy check, we will not include Step 5 for the enforced phase transition at emerged triple points.
 \end{rem}

{\blue

\subsection{Unconditional stability of the projection method for the moving contact line and the motion by mean curvature}\label{sec_stability}
In this section, we show the unconditional  stability of the projection method for the moving contact line and the motion by mean curvature. To focus on the key difficulties due to the moving domain and the obstacle, we present an unconditional stability analysis for  the case the droplet is placed on a horizontal plane without gravity and volume constraint. We will first present a projection method with a small modification for the stretching term.

\subsubsection{A projection method for  the moving contact line and the motion by mean curvature}

To focus on the moving contact line and the obstacle problem, we first present a simplified  projection method for droplets  placed on a horizontal plane without gravity and volume constraint.\\
Step 1.  Explicit boundary updates. We update $a^{n+1}, b^{n+1}$ using 
\begin{equation}
\begin{aligned}\label{end_a}
\frac{a^{n+1}-a^n}{\Delta t}&= \sigma  +\frac{1 }{\sqrt{1+ (\pt_x h^n)^2}}\Big| _{a^n} , \quad \frac{b^{n+1}-b^n}{\Delta t}&= -\sigma  -\frac{1 }{\sqrt{1+ (\pt_x h^n)^2}}\Big|_{b^n}.
\end{aligned}
\end{equation}

Step 2. Rescale $h^n$ from $[a^n, b^n]$ to $[a^{n+1}, b^{n+1}]$ such that
\begin{equation}\label{inter-u-0}
    h^{n*}(x^{n+1}):= h^n(x^n) = U(Z), \quad x^{n+1}\in[a^{n+1}, b^{n+1}],
\end{equation}
where $Z\in[0,1]$ is a fixed domain variable satisfying
 \begin{equation}\label{changing1}
Z= \frac{x^n-a^n}{b^n -a^n} = \frac{x^{n+1}-a^{n+1}}{b^{n+1}-a^{n+1}}.
\end{equation}


Step 3. Capillary surface updates without impermeable obstacle constraint. Update $\tilde{ h}^{n+1}$ and $\lambda^{n+1}$ implicitly.
\begin{equation}\label{tm313_sb}
\begin{aligned}
&\frac{\beta}{\sqrt{1+ (\pt_x \tilde{h}^{n+1})^2}} \frac{\tilde{ h}^{n+1}-h^{n*}}{\Delta t}
= \frac{\pt }{\pt x}\left( \frac{\pt_x \tilde{ h}^{n+1} +\pt_x h^{n*}}{ \sqrt{1+ (\pt_x \tilde{h}^{n+1})^2}+ \sqrt{1+ (\pt_x h^{n*})^2}} \right), \\
&\tilde{ h}^{n+1}(a^{n+1})=0, \quad \tilde{ h}^{n+1}(b^{n+1})=0.
\end{aligned}
\end{equation}
where the independent variable  is $x^{n+1}\in(a^{n+1},b^{n+1})$.

Step 4. Projection due to the impermeable obstacle.
\begin{equation}
h^{n+1} = \max\{\tilde{h}^{n+1}, 0\}.
\end{equation}

\subsubsection{Unconditional stability  for the projection method} Now we prove a proposition for the unconditional stability  of the simplified projection method above.
\begin{prop}\label{prop_stability}
For any $T>0$, let $\Delta t$ be the time step and  $N=\frac{T}{\Delta t}$. Suppose $h^n$, $h^{n^*}$ and $\tilde{h^n}$ are obtained from the above projection method.  Then we have the following stability estimates
\begin{enumerate}[(i)]
\item for endpoints
 \begin{equation}\label{stable_ab0}
\begin{aligned}
a^0 + \sigma T \leq {\blue a^{n}} \leq a^0 + (1+\sigma)T, \qquad b^0 - (1+\sigma) T \leq {\blue b^{n} } \leq b^0 - \sigma T;
\end{aligned}
\end{equation}
\item for the capillary surface
\begin{equation}\label{stable_cap}
\max_{1\leq n\leq N} \int_{a^n}^{b^n} \sqrt{1+|\nabla h^{n}|^2} \ud x +\frac{\beta}{\Delta t} \sum_{n=0}^{N-1}  \int_{a^{n+1}}^{b^{n+1}} \frac{|\tilde{h}^{n+1} -h^{n*}|^2}{ \sqrt{1+ (\pt_x \tilde{h}^{n+1})^2} } \ud x \leq  8 T + 2 \int_{a^0}^{b^0} \sqrt{1+ (\pt_x h^{0})^2}  \ud x.
\end{equation}
\end{enumerate}
\end{prop}

\begin{proof}

First, we give the stability estimates for endpoints $a^n, b^n$.

From \eqref{end_a}, we know 
\begin{equation}\label{stable_ab}
\sigma \Delta t \leq a^{n+1}-a^n \leq( \sigma+1) \Delta t, \quad  -(\sigma+1)\Delta t \leq  b^{n+1}-b^n\leq -\sigma \Delta t.
\end{equation}
Then telescoping shows \eqref{stable_ab0}.

Second, we give the stability estimates for $\tilde{h}^n$.

From \eqref{changing1}
and elementary calculations we list the following expressions in terms of $Z$ variable
\begin{align}\label{changing1}
&h^n(x^{n}) = U(Z)=h^{n*}(x^{n+1}), \quad \pt_x h^n(x^n) = \frac{\pt_Z U(Z)}{b^n-a^n}, \quad 
\pt_x h^{n*} (x^{n+1}) =   \frac{\pt_Z U(Z) }{b^{n+1}-a^{n+1}}.
\end{align}
Multiplying the equation \eqref{tm313_sb}
by $v:= h^{n+1}-h^{n*}$ and integrating from $a^{n+1}$ to $b^{n+1}$, we obtain
\begin{equation}
\begin{aligned}
\frac{\beta}{\Delta t} \int_{a^{n+1}}^{b^{n+1}} \frac{|\tilde{h}^{n+1} -h^{n*}|^2}{ \sqrt{1+ (\pt_x h^{n+1})^2}} \ud x &+ \int_{a^{n+1}}^{b^{n+1}} \frac{|\pt_x \tilde{h}^{n+1}|^2 - |\pt_x h^{n*}|^2}{ \sqrt{1+ (\pt_x \tilde{h}^{n+1})^2}+ \sqrt{1+ (\pt_x h^{n*})^2} } \ud x =0.
\end{aligned}
\end{equation}
This gives
\begin{equation}\label{tal1}
\begin{aligned}
\frac{\beta}{\Delta t} \int_{a^{n+1}}^{b^{n+1}} \frac{|\tilde{h}^{n+1} -h^{n*}|^2}{ \sqrt{1+ (\pt_x \tilde{h}^{n+1})^2}} \ud x &+ \int_{a^{n+1}}^{b^{n+1}} \bbs{ \sqrt{1+ (\pt_x \tilde{h}^{n+1})^2}- \sqrt{1+ (\pt_x h^{n*})^2} } \ud x  =0.
\end{aligned}
\end{equation}

We now  prove the following claim
\begin{equation}\label{clm}
\begin{aligned}
 \int_{a^{n+1}}^{b^{n+1}} \sqrt{1+ (\pt_x h^{n*})^2}  \ud x -  \int_{a^{n}}^{b^{n}}  \sqrt{1+ (\pt_x h^{n})^2} \ud x  \leq  4\Delta t.
\end{aligned}
\end{equation}
\begin{proof}[Proof of the claim \eqref{clm}]
Frist, from \eqref{stable_ab} and $|\sigma|<1$, we know
\begin{equation}\label{4dt}
\left| (b^{n+1}- a^{n+1}) - (b^n -a^n)  \right| \leq 4 \Delta t. 
\end{equation}

Second, by changing variables \eqref{changing1}, we have
\begin{equation}
\begin{aligned}
 &\int_{a^{n+1}}^{b^{n+1}} \sqrt{1+ (\pt_x h^{n*})^2}  \ud x -  \int_{a^{n}}^{b^{n}}  \sqrt{1+ (\pt_x h^{n})^2} \ud x  \\
 =&\int_0^1  \sqrt{(b^{n+1}-a^{n+1})^2+ \bbs{\pt_Z U(Z)}^2}   \ud Z -  \int_0^1 \sqrt{(b^n-a^n)^2+ (\pt_Z U(Z))^2}  \ud Z
\\
=& \int_0^1  \frac{(b^{n+1}-a^{n+1})^2 - (b^n-a^n)^2  }{\sqrt{(b^{n+1}-a^{n+1})^2+ \bbs{\pt_Z U(Z)}^2 } + \sqrt{(b^n-a^n)^2+ (\pt_Z U(Z))^2} }  \ud Z
 \end{aligned}
\end{equation} 
Therefore, from \eqref{4dt}, we know
\begin{equation}
\begin{aligned}
 &\left|\int_{a^{n+1}}^{b^{n+1}} \sqrt{1+ (\pt_x h^{n*})^2}  \ud x -  \int_{a^{n}}^{b^{n}}  \sqrt{1+ (\pt_x h^{n})^2} \ud x \right|  \\
 \leq &  \frac{  \left|(b^{n+1}-a^{n+1})^2 - (b^n-a^n)^2 \right| }{\sqrt{(b^{n+1}-a^{n+1})^2} + \sqrt{(b^n-a^n)^2}}    \leq  4\Delta t.
 \end{aligned}
\end{equation}
\end{proof}

From this claim, \eqref{tal1} becomes
\begin{equation}\label{tal2}
\begin{aligned}
\frac{\beta}{\Delta t} \int_{a^{n+1}}^{b^{n+1}} \frac{|\tilde{h}^{n+1} -h^{n*}|^2}{ \sqrt{1+ (\pt_x \tilde{h}^{n+1})^2} } \ud x + \int_{a^{n+1}}^{b^{n+1}}  \sqrt{1+ (\pt_x \tilde{h}^{n+1})^2} \ud x- \int_{a^n}^{b^n} \sqrt{1+ (\pt_x h^{n})^2}  \ud x  \leq 4\Delta t.
\end{aligned}
\end{equation}

Third, we give the estimate for $h^n$.

We prove by induction. For $n=0$, from \eqref{tal2} and the Dirichlet boundary condition
we know $\tilde{h}^{1} \in W^{1,1}(a^1, b^1)$.  Thus
\begin{equation}
\pt_x h^{1} = \pt_x (\tilde{h}^1_{+}) = \text{sgn} (\tilde{h}^1) \pt_x \tilde{h}^1 \in L^1(a^1, b^1).
\end{equation}
This implies 
\begin{equation}
\int_{a^1}^{b^1} \sqrt{1+|\pt_x h^1|^2} \ud x \leq  \int_{a^1}^{b^1} \sqrt{1+|\pt_x \tilde{h}^1|^2} \ud x.
\end{equation}
Then by induction, \eqref{tal2} becomes
\begin{equation}\label{tal3}
\begin{aligned}
\frac{\beta}{\Delta t} \int_{a^{n+1}}^{b^{n+1}} \frac{|\tilde{h}^{n+1} -h^{n*}|^2}{ \sqrt{1+ (\pt_x \tilde{h}^{n+1})^2} } \ud x + \int_{a^{n+1}}^{b^{n+1}}  \sqrt{1+ (\pt_x{h}^{n+1})^2} \ud x- \int_{a^n}^{b^n} \sqrt{1+ (\pt_x h^{n})^2}  \ud x  \leq 4\Delta t.
\end{aligned}
\end{equation}
By telescoping, we obtain
\begin{equation}
 \int_{a^{n+1}}^{b^{n+1}}  \sqrt{1+ (\pt_x h^{n+1})^2} \ud x \leq 4T + \int_{a^0}^{b^0} \sqrt{1+ (\pt_x h^{0})^2}  \ud x
\end{equation}
and thus we conclude \eqref{stable_cap}.

\end{proof}

}

{\blue 

\subsection{Convergence analysis of the projection method with pinning contact lines}\label{sec_conv}

In this section, we give  convergence analysis for the projection method under the pinning contact line assumption.
On the one hand, the pinning (or sticking) effect  is an important observed phenomenon in most droplet wetting applications \cite{dussan1985ability, schaffer1998dynamics, de2013capillarity}.   On the other hand, the convergence analysis of the projection method for the original moving contact line problem is very challenging. In the following subsections, we first represent the PVI solution with pinning contact line as a nonlinear semigroup solution,  then introduce the associated projection method and its resolvent representation, and finally give the convergence analysis of the projection method using a nonlinear version of  Trotter-Kato's product formula.

\subsubsection{Pinning contact line and the associated PVI}
In order to work in a Hilbert space $X:=L^2(D)$, we clarify the following two assumptions on the Rayleigh dissipation functional $Q$.
Recall \eqref{QQ22}.  First, we replace the dissipation due to the motion of the capillary surface as $\zeta\int_{D(t)} |v|^2 \ud x \ud y$ so 
\begin{equation}
\tilde{Q}:=\ssr \int_{\Gamma(t)} |v_{\cl}|^2 \ud s  + \zeta\int_{D(t)} |v|^2 \ud x \ud y.
\end{equation} 
Second, we assume the friction coefficient $\ssr\to +\8$, which leads to the pinning boundary condition $v_{\cl}=0$.
Then the PVI \eqref{PVI_new} becomes  the following parabolic obstacle problem in a fixed domain $D$.  Denote $K:=\{u\in L^2(D); u\geq 0\}$.
\begin{equation}\label{PVI_class}
\begin{aligned}
& \text{ Find } u(t) \in K \cap C_0(D) , \,\, \lambda(t) \in \mathbb{R} \,\, \text{ such that }\\
& \zeta\int_{D} \pt_t u (\tilde{ v}-  u) \ud x \ud y \geq \int_{D} F (\tilde{v}- u ) \ud x \ud y, \quad  \forall  \tilde{v}\in K\cap C_0(D),\\
& \int_{D} u \ud x\ud y =V,
\end{aligned}
\end{equation}
where $F = \gamma_{\nlg}\nabla \cdot \bbs{\frac{\nabla u}{\sqrt{1+|\nabla u|^2}}}  -\rho g u + \lambda .$

\subsubsection{The sum of two maximal monotone operators and the nonlinear semigroup solution}\label{sec_semi}
Choose any fixed capillary profile $u^*\in C_0^2(D)$ satisfying $\int_{D} u^* \ud x \ud y=V, \, u^*\geq 0$ as a reference. We introduce  a Hilbert space
$$X_1:= \{v\in L^2(D); \int_D v \ud x \ud y = 0\}.$$
Define a functional $E$ from $X_1$ to $[0, +\8]$
\begin{equation}
E(v):= \F(v+u^*)=\gamma_{\nlg} \int_D \sqrt{1+|\nabla (v+u^*)|^2} \ud x \ud y + \rho g \int_D \frac{(v+u^*)^2}{2} \ud x \ud y.
\end{equation}
First, it is easy to verify that $E$ is a proper, convex, lower semi-continuous functional on $X_1$.  Second, in $X_1$,  we compute the G\^ateaux derivative of $E$: for any $\tilde{v}\in X_1$,
\begin{equation}
\frac{\ud }{\ud \eps} \Big|_{\eps = 0}E(v+\eps \tilde{v}) = \la -\gamma_{\nlg}\nabla \cdot \bbs{\frac{\nabla (v+u^*)}{\sqrt{1+|\nabla (v+u^*)|^2}}}  +\rho g (v+u^*) - \lambda, \tilde{v} \ra
\end{equation}
for any constant $\lambda\in \mathbb{R}$. Indeed, any constants   are zero element in $(X_1)'$.
Thus we know the subdifferential of $E$ in $X_1$ is single-valued and agrees with the  G\^ateaux derivative; denoted as 
\begin{equation}
\nabla E (v) :=-\gamma_{\nlg}\nabla \cdot \bbs{\frac{\nabla (v+u^*)}{\sqrt{1+|\nabla (v+u^*)|^2}}}  +\rho g (v+u^*) - \lambda.
\end{equation}

Since $E$ is a proper, convex, lower semi-continuous functional on $X_1$,  we know $\nabla E$ is a maximal monotone operator which generates a nonlinear $C_0$-semigroup; symbolically denoted as  $S_1(t):= e^{-t \nabla E}$. The semigroup solution $u(t)=v(t)+u^*$ satisfies
the following governing equations
\begin{equation}
\begin{aligned}
\zeta \pt_t u = \gamma_{\nlg}\nabla \cdot \bbs{\frac{\nabla u}{\sqrt{1+|\nabla u|^2}}}  -\rho g u + \lambda  \quad \text{ in }D,\\
u = 0 \quad  \text{ on }\pt D,\\
\int_{D} u \ud x \ud y = V.
\end{aligned}
\end{equation}
Notice the left hand side in the first equation above is vertical velocity instead of normal velocity of the capillary surface due to the special choice of $\tilde{Q}$.

For the obstacle problem \eqref{PVI_class}, we need to introduce an indicator functional $I_K$ for the convex subset
\begin{equation}
K:=\{v\in X_1; u=u^* +v \geq 0\}\subset X_1,
\end{equation}
which is indeed a convex cone.
Denote $A:=\nabla E$, $B:=\pt I_K$. Since $A+B= \nabla E+ \pt I_K$ is a maximal monotone operator in $X_1$, $A+B$ generates a strongly continuous semigroup on $[0,+\8)$ of  contractions  \cite{komura1967nonlinear, brezis1973operateurs}, symbolically denoted as $S(t):= e^{-t(\nabla E+ \pt I_K)}$.
For any $u_0= v_0+u^*, \, v_0\in K$, the unique mild solution to \eqref{PVI_class} is given by  \cite{brezis1973operateurs}
\begin{equation}\label{mild-solu}
u(t) = e^{-t (\nabla E+ \pt I_K)} u_0:=u^*+ e^{-t (\nabla E+ \pt I_K)} v_0 .
\end{equation}

\subsubsection{Projection method and resolvent representation}
To numerically solve the obstacle problem \eqref{PVI_class} generated by $A+B$, the projection method (the splitting method) is a natural and efficient method. 
Its convergence will be shown below by using a nonlinear version of  Trotter-Kato's product formula \cite{Kato_Masuda_1978}.

We first present the  projection method for \eqref{PVI_class} in terms of the abstract operators  $\nabla E$ and $\pt I_K$ as follows. For $\tau:= \frac{T}{n}$, we use the projection method to construct an approximation $u^k$ of  $u(t^k)$ with $t^k=k \tau, \, k=0,1, \cdots, n$ below.
\begin{enumerate}[Step (i)]
\item For $u^k=v^k+u^*$, $v^k\in K$, find
\begin{equation}
\tilde{u}^{k+1} = u^*+ \argmin_{v\in X_1} E(v)+ \frac{\|v-v^k\|^2}{2\tau};
\end{equation}
which is symbolically given by 
\begin{equation}
\tilde{u}^{k+1} = (I+\tau\nabla E)^{-1} u^k;
\end{equation}
\item Find $u^{k+1}$ and  $\lambda$ satisfying
\begin{equation}\label{proj340}
\left\{
\begin{array}{cc}
u^{k+1} = \max\{ \tilde{u}^{k+1} + \lambda ,\,  0\},\\
\int_{D} u^{k+1} \ud x \ud y = V.
\end{array}
\right.
\end{equation}
 
\end{enumerate}
Indeed, we have the following lemma characterizing the projection operator in Step (ii).
\begin{lem}\label{lem_proj}
Given $\tilde{u}^{k+1}$ and $\tilde{v}^{k+1}=\tilde{u}^{k+1}-u^*$,  the following problems are equivalent:
\begin{enumerate}[(i)]
\item Find $ u^{k+1}=v^{k+1}+u^*$ with $v^{k+1}$ satisfying
\begin{equation}\label{pvi343}
\left\{
\begin{array}{cc}
v^{k+1}\in K,\\
\la\tilde{v}^{k+1} - v^{k+1} + \lambda , \phi - v^{k+1} \ra \leq 0 \quad \forall \phi \in K;
\end{array}
\right.
\end{equation}
\item Find $ u^{k+1}=v^{k+1}+u^*$ with $v^{k+1}$ satisfying
\begin{equation}\label{res341}
\tilde{v}^{k+1} - v^{k+1} \in \tau \pt I_K (v^{k+1}), \quad \forall \tau>0;
\end{equation}
\item Find $ u^{k+1}=v^{k+1}+u^*$ with $v^{k+1}$ satisfying $v^{k+1}=\proj_{K}(\tilde{u}^{k+1}-u^*)$;
\item Find $u^{k+1}$ and $\lambda$ satisfying \eqref{proj340};
\end{enumerate} 
\end{lem}
\begin{proof}

First, we prove (i) is equivalent to (ii).

The subdifferential of $I_K$ at $v\in K$ in $X_1$ is given by
\begin{equation}
f\in \tau \pt I_K(v) \llra 0 \geq \la f+\lambda, \phi - v \ra \quad \forall \phi \in K.
\end{equation}
where $\lambda$ is any constants.

Thus $\pt I_K $ is convex cone and  
\begin{equation}
\begin{aligned}
& \tilde{u}^{k+1} - u^{k+1} = \tilde{v}^{k+1} - v^{k+1} \in \tau \pt I_K(v^{k+1})\\
 \llra & 
\left\{
\begin{array}{cc}
v^{k+1}\in K,\\
\la\tilde{v}^{k+1} - v^{k+1} + \lambda , \phi - v^{k+1} \ra \leq 0 \quad \forall \phi \in K.
\end{array}
\right.
\end{aligned}
\end{equation}

Second, the equivalence between (i) and (iii) can be  verified by introducing a Lagrange multiplier $\lambda$ to enforce the volume constraint.
Define 
$$\tilde{K}:=\{u\in L^2(D); u\geq 0\}, \quad  L(u,\lambda):=\frac12\|u-\tilde{u}^{k+1}\|^2 - \lambda (\int_D u \ud x \ud y - V).$$
Thus the projection is
\begin{equation}
u^{k+1}=\proj_K(\tilde{v}^{k+1}) + u^*,      \quad (u^{k+1}, \lambda^*)=   \arg\min_{u\in \tilde{K}, \lambda\in \mathbb{R}} L(u, \lambda).
\end{equation}

Thus $(u^{k+1}, \lambda^*)$ satisfies
\begin{equation}
\left\{
\begin{array}{cc}
&\int_{D} u^{k+1} \ud x \ud y = V,\\
&L(u^{k+1}, \lambda^*) \leq L(u^{k+1} + \eps(\phi-u^{k+1}) , \lambda^*), \quad  \forall0<\eps <1, \tilde{\phi} \in \tilde{K}.
\end{array}
\right.
\end{equation}
Taking $\eps\to 0^+$, we have
\begin{equation}\label{tm_347}
\left\{
\begin{array}{cc}
&\int_{D} u^{k+1} \ud x \ud y = V,\\
&\la u^{k+1}-\tilde{u}^{k+1}, \tilde{\phi} - u^{k+1} \ra - \lambda^* \int_{D} (\tilde{\phi} - u^{k+1}) \geq 0, \quad \tilde{\phi} \in \tilde{K}.
\end{array}
\right.
\end{equation}
which is exactly (i).

Third, to prove the equivalence between (iv) and (i), we only need to prove the equivalence between \eqref{proj340} and \eqref{tm_347}  since $u^*$ is the reference state satisfying $\int_D u^* \ud x \ud y=V$. 
Define 
$$f(\lambda):=\int_D \max\{\tilde{u}^{k+1} + \lambda, 0\} \ud x \ud y - V,$$
which is an increasing function with respect to $\lambda$.
It is easy to verify
\begin{align*}
f(0) \geq \int_D \tilde{u}^{k+1} \ud x \ud y - V = 0
\end{align*}
while for $\lambda_{m}=-\max(\tilde{u}^{k+1})\leq 0$,
\begin{align*}
f(\lambda_m) = -V \leq 0.
\end{align*}
Thus there exists a unique $\lambda^*$ such that
\begin{equation}
\left\{
\begin{array}{cc}
u^{k+1} = \max\{ \tilde{u}^{k+1} + \lambda^* , 0\},\\
\int_{D} u^{k+1} \ud x \ud y = V.
\end{array}
\right.
\end{equation}
Then the equivalence of this with \eqref{tm_347} is directly concluded \cite[p. 27]{Kind00}.

\end{proof}

Rewrite  \eqref{res341}  as 
\begin{equation}
v^{k+1} = (I+\tau \pt I_K)^{-1} (\tilde{v}^{k+1}),
\end{equation}
and symbolically
\begin{equation}
u^{k+1} = (I+\tau \pt I_K)^{-1} \tilde{u}^{k+1}:=u^* + (I+\tau \pt I_K)^{-1} (\tilde{v}^{k+1}).
\end{equation}
In summary, we have
\begin{equation}
u^{k+1}= (I+\tau \pt I_K)^{-1} (I+\tau \nabla E)^{-1} u^k.
\end{equation}
Recall $A=\nabla E$, $B=\pt I_K$.
The corresponding resolvent operators of $A$ and $B$ are denoted as $J_\tau^A := (I+\tau \nabla E)^{-1} $ and $J_\tau^B := (I+\tau \pt I_K)^{-1}$ respectively.
The projection method above can be recast as
\begin{equation}\label{kato1}
u^{k+1}=J_\tau^B J_\tau^A u^k.
\end{equation}
Next, to obtain an approximated solution $\bar{u}_\tau(t)$ at any $t\in[0,T]$ for any time step $\tau>0$, we use the piecewise constant interpolation from $u^k$  such that
\begin{equation}
\bar{u}_\tau(t)\equiv u^k, \quad t\in[k \tau, (k+1) \tau),
\end{equation}
which is equivalent to
\begin{equation}\label{scheme2}
\bar{u}_\tau (t):= (J_{\tau}^B J_{\tau}^A)^{[\frac{t}{\tau}]} u_0, \quad t\in[0,T].
\end{equation}

\subsubsection{Convergence theorem}
With all the preparations above, we apply a nonlinear version of  Trotter-Kato's product formula \cite{Kato_Masuda_1978} to prove the convergence for solving PVI \eqref{PVI_class} using the projection method. 
\begin{thm}[Convergence of the projection method]\label{thm1}
Let $t\in[0,T]$ and $u(t)$ given in \eqref{mild-solu} be the mild solution of \eqref{PVI_class}. Let $\bar{u}_\tau(t)$ given in \eqref{scheme2} be the numeric solution obtained from the projection method with time step $\tau$. Then we have the convergence
\begin{equation}
\lim_{\tau\to 0} \sup_{t\in[0,T]} \|\bar{u}_\tau(t) - u(t)\|=0.
\end{equation} 
\end{thm}
\begin{proof}

First, since $A+B= \pt(E+I_K)$ is a maximal monotone operator in $X$, $A+B$ generates a strongly continuous semigroup on $[0,+\8)$ of  contractions, denoted symbolically as $S(t):= e^{-t\pt(E+I_K)}$.
For any $u_0\in K$, the mild solution to \eqref{maineq} is given by 
\begin{equation}
u(t) = e^{-t\pt(E+I_K)} u_0.
\end{equation}

Second, recall the resolvent of $A$ and $B$ are $J_\tau^A = (I+\tau \nabla E)^{-1} $ and $J_\tau^B = (I+\tau \pt I_K)^{-1}$ respectively. We use Trotter-Kato's product formula \cite{Kato_Masuda_1978} to prove
\begin{equation}\label{con1}
u^n(t):= \bbs{J_{\frac{t}{n}}^B J_{\frac{t}{n}}^A }^n u_0 \to u(t), \quad \text{ as }n\to +\8,\quad \text{ uniformly for } t\geq 0. 
\end{equation}
To see this, in \cite{Kato_Masuda_1978} we take $U_1:=J_{\frac{t}{n}}^A$, $U_2:=J_{\frac{t}{n}}^B$. Then by \cite[Example 2.3]{Kato_Masuda_1978}, we know $U_1$ are a nice $E$-family  with index $\gamma\geq 2$ and  $U_2$ are a nice $I_K$-family  with index $\gamma\geq 2$. Hence the condition (i) in \cite[Theorem]{Kato_Masuda_1978} holds, which gives the claim  \eqref{con1}.

Finally, for the projection scheme for $t\in[0,T]$ with time step $\tau=\frac{T}{n}$, the piecewise constant interpolation in $[k\tau, (k+1)\tau)$ is given by
\begin{equation}
\bar{u}_\tau(t)= (J_{\tau}^B J_{\tau}^A)^{[\frac{t}{\tau}]} u_0, \quad t\in[0,T],
\end{equation}
where $[a]$ is the integer part of real number $a$. Since $\tau [\frac{t}{\tau}] \to t$ as $\tau\to 0$, we know \begin{equation}
u(\tau [\frac{t}{\tau}]) \to u(t)
\end{equation}
due to continuous semigroup property. Therefore we conclude 
\begin{equation}
\|\bar{u}_\tau(t) - u(t)\|\leq\|(J_{\tau}^B J_{\tau}^A)^{[\frac{t}{\tau}]} u_0-u(\tau [\frac{t}{\tau}])\|+ \|u(\tau [\frac{t}{\tau}]) - u(t)\| \to 0
\end{equation}
as $\tau\to 0$ uniformly in $t\in[0,T]$.

\end{proof}

}

\section{Simulations for merging and splitting of droplets}\label{sec_simu}
In this section, we first give an  accuracy check for the numerical scheme by constructing a projected  triple Gaussian capillary surface  and then demonstrate two typical examples using the projection scheme proposed in Section \ref{sec_scheme}. The first example is  splitting of one big droplet into two droplets when placed on an inclined groove-textured substrate. The second example is   merging of two droplets in a Utah teapot, which is compared to  independent dynamics of two droplets in the teapot separately.

\subsection{Accuracy check with a projected  triple Gaussian profile } In this subsection, to check the order of accuracy of the projection method, we construct a special example by a projected  triple Gaussian function. We point out we only check the 1st order  accuracy of the projection method for solving PVI \eqref{PVI_class}, i.e., Steps 1-4 in Section \ref{sec_1st}. In other words, when checking the 1st order accuracy of the projection method, we always regard the projected profile as a two-phase interface without detecting the  splitting point (Step 5 in Section \ref{sec_1st}).

We use   the initial endpoint $b_0= 1.3$ and a projected  triple Gaussian as initial capillary surface (green line in Figure \ref{fig_triple}) 
\begin{equation}\label{u0}
 u_0(x) = \max\{2\bbs{e^{-(x+1)^2}+e^{-16x^2}+e^{-(x-1)^2} - e^{-(b_0+1)^2}-e^{-16b_0^2}-e^{-(b_0-1)^2}},\,\, 0\} \quad  \text{ for }x\in[-b_0, b_0].
\end{equation}
The physical parameters in \eqref{tm313} for the first order scheme in Section \ref{sec_1st}  are 
\begin{equation}
 \kappa = 20 , \quad  \beta = 2, \quad  \sigma= - \cos\frac{\pi}{3}=-\frac12, \quad w(x)\equiv 0,
\end{equation}
where $\theta_Y=\frac{\pi}{3}$ is the Young's angle.
Choose a final time   $T=0.2$. We numerically compute an exact solution  $b(T)=1.238369607448268, \, \theta_{\cl}(T)=0.493880485215342$ with uniform time steps $\Delta t = \frac{T}{M},\, M=1280$ and uniform  $N=8*M$ moving  grid points  in $[a(t), b(t)]$.

We show below the accuracy check  for the first order scheme in Section \ref{sec_1st}  in Table \ref{table_1st}.    We use the same physical parameters and initial data   in the first  order schemes. For several $M_n$ listed in the tables, we take time step as $\Delta t= \frac{T}{M_n}$ and  moving grid size $\Delta x= \frac{b(t)-a(t)}{N_n}$ with $N_n=8M_n$. The absolute error $e_n =|b^{M_n}-b(T)|$ (resp. $e_n=|\theta_{\cl}^{M_n} - \theta_{\cl}(T)|$) between numeric solutions and the numerically computed exact contact point $b(T)$ (resp. exact contact angle $\theta_{\cl}(T)$) are listed in the second column (resp. sixth column) of Table \ref{table_1st}. The maximal norm error $e_n =\|u^{M_n}-u(\cdot,T)\|_{\ell^\8}$ between numeric solutions and  the numerically computed exact capillary profile $u$ is listed in the fourth column of Table \ref{table_1st}.
The corresponding order of accuracy $\alpha = \frac{\ln(e_n/e_{n+1})}{\ln(M_{n+1}/M_{n})}$ is listed in the last column of the tables. For $M=320,$ the time evolution of the capillary surface is shown in Figure \ref{fig_triple} (upper) with red lines at equal time intervals. A zoom in plot showing the capillary surface for the small bump profile at $x=0$ approaches zero at the early stage of the evolution. We also track the time history of two contact angles  upto the final time $T=0.2$; see Figure \ref{fig_triple} (lower).

\begin{table}[ht]
\caption{Accuracy check: 1st order projection scheme in Section \ref{sec_1st}  v.s numerically computed exact  solution to \eqref{tm313} using time steps $M=1280$ and spatial grid points $N=8*M$. Parameters: $T=0.2$, $\kappa=20$, $\theta_Y=\frac{\pi}{3}$,  $b_0=1.3$, time step $\Delta t= \frac{T}{M}$, $M$ listed on the table, moving grid size $\Delta x= \frac{b(t)-a(t)}{N}$,  $N=8M$. Absolute errors $e_n$ for $b$, $u$ and $\theta_{\cl}$ are computed by comparing with exact solutions.}\label{table_1st}
\begin{tabular}{|c|c|c|c|c|c|c|}
\hline 
\,  &  \multicolumn{6}{c|}{1st order scheme}  \\ 
\hline 
M & Error of $b$  & Order  &Error of $u$  & Order & Error of $\theta_{\cl}$  & Order   \\ 
\hline \hline
$20$ & $   \num{2.5961E-3}$ & \,& $\num{1.2592E-3}$ &  & $\num{6.3988E-3}$ & \\
$40$ & $  \num{1.3571E-3}$ & 0.936& $\num{6.3389E-4}$ & $0.990$ & $\num{3.0201E-3}$& $1.083$  \\
$80$ & $ \num{6.8157E-04}$ &0.994 & $ \num{3.0961E-4}$ & $1.034$ & $\num{1.4511E-3}$& $1.057$  \\
$160$ & $ \num{3.2538E-4}$ &1.067  & $ \num{1.4555E-4}$ & $1.089$ & $\num{6.7734E-4}$& $1.099$ \\
$320$ & $ \num{1.4151E-4}$ &1.201 & $ \num{6.2874E-05}$ & $1.211$ & $\num{2.8999E-4}$& $1.224$ \\
\hline
\end{tabular} 
\end{table}

\begin{figure}
\includegraphics[scale=0.4]{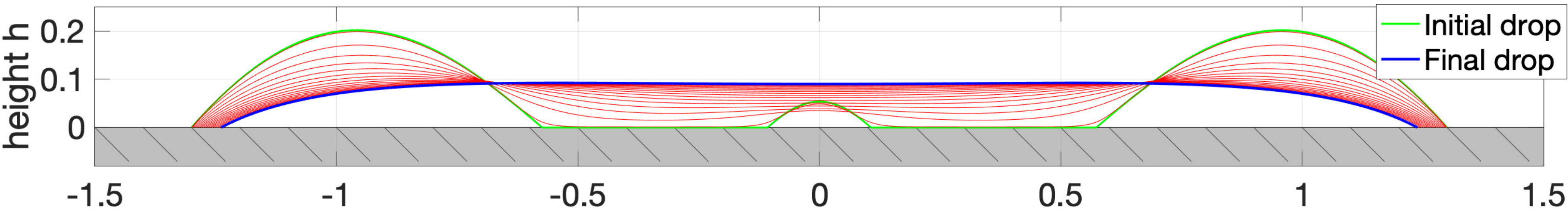}

 \includegraphics[scale=0.32]{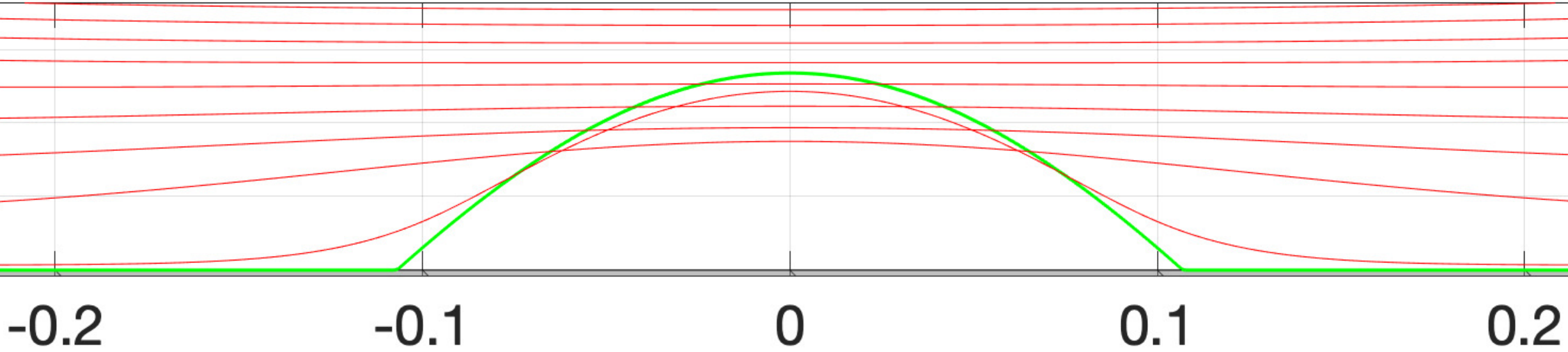}\hspace{0.01in}
 \includegraphics[scale=0.22]{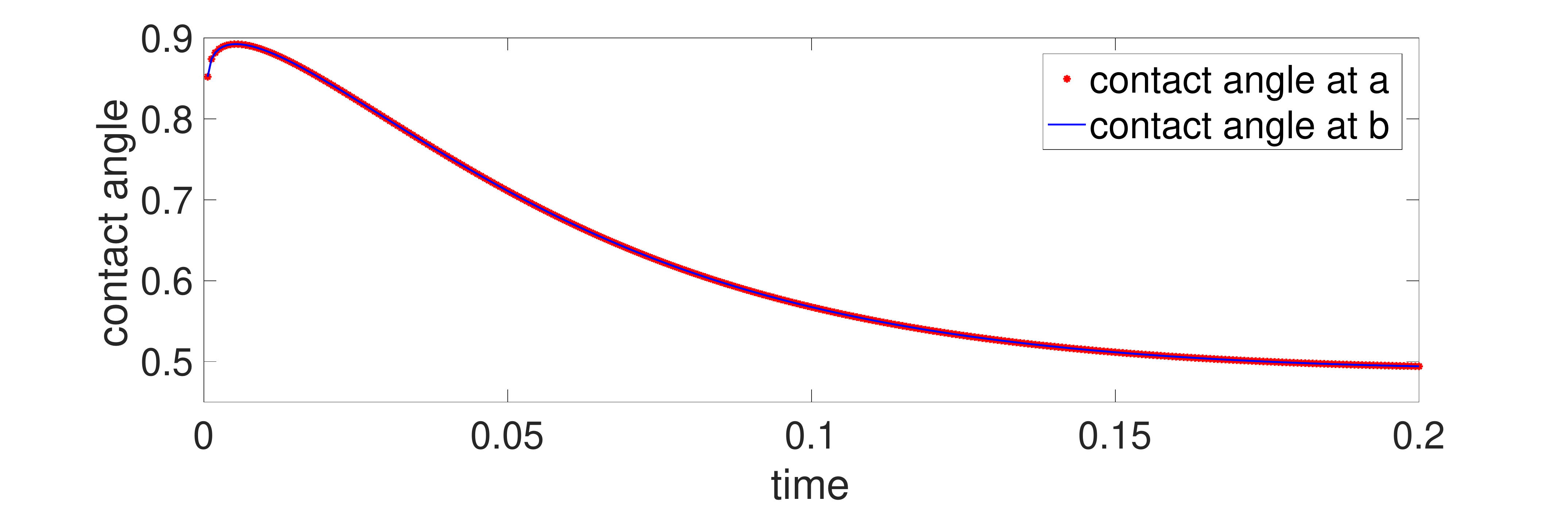}  
\caption{The time evolution of the capillary surface $u$ starting from a projected triple Gaussian function \eqref{u0}, shown in green line. (Upper) Capillary surface at equal time intervals are shown in red line and the profile at final time $T=0.2$ is shown in blue line. (Lower left) Zoom in plot near $x=0$ showing $u$ approaches zero at an early stage; (Lower right) Dynamic contact angle tracking w.r.t time.}\label{fig_triple}
\end{figure}


\subsection{Computations}
Now we use the projection scheme in Section \ref{sec_scheme} to simulate two challenging examples including the splitting and merging of droplets on different impermeable substrates.
\subsubsection{Example 1: Splitting of one droplet on an inclined  groove-textured substrate.}

We take a typical groove-textured substrate
\begin{align}
w(x)=A(\sin(kx)+\cos(2kx))^2, \, A=0.1,\, k=2.5.\label{w1_r}
\end{align}
This is an impermeable obstacle where phase transitions happen when the droplet touches the obstacle. Thus at the touching point, after one detects the phase transition,  one droplet will split into two independent droplets with their own PVI \eqref{wet-eq-r}.
To demonstrate those phenomena, 
we take the physical parameters as $\kappa=1, \, \beta=0.1$, effective inclined angle $\theta_0=0.3$ and initial droplet as 
\begin{equation}\label{h0_r1}
h(x,0)=0.1(x-a(0))(b(0)-x)+w(a(0))+\frac{[w(b(0))-w(a(0))](x-a(0))}{b(0)-a(0)}
\end{equation}
 with initial endpoints $a(0)=-2.1,\, b(0)=3.1$ as shown in Fig \ref{fig_r1r2} using a green line. The corresponding effective Bond number can be calculated as in Remark \ref{rem1} with effective inclined angle $\theta_0=0.3$, $\Bo=0.5712$. 
We take final time as $T=1$ with time step $\Delta t = 0.005$ and use $N=200$ moving grids uniformly in $(a(t), b(t))$ in the projection scheme.  With relative adhesion coefficient $\sigma=-0.52$, in Fig. \ref{fig_r1r2}, we show  the dynamics of the droplet on groove-textured surface $w(x)$ in \eqref{w1_r} at equal time intervals using thin red lines. The splitting time detected is $T_s=0.035$ with threshold $\eps=0.075$ and the two generated droplets keep moving independently until the final time $T=1$ with the final profiles shown in solid blue lines. 
\begin{figure}
 \includegraphics[scale=0.34]{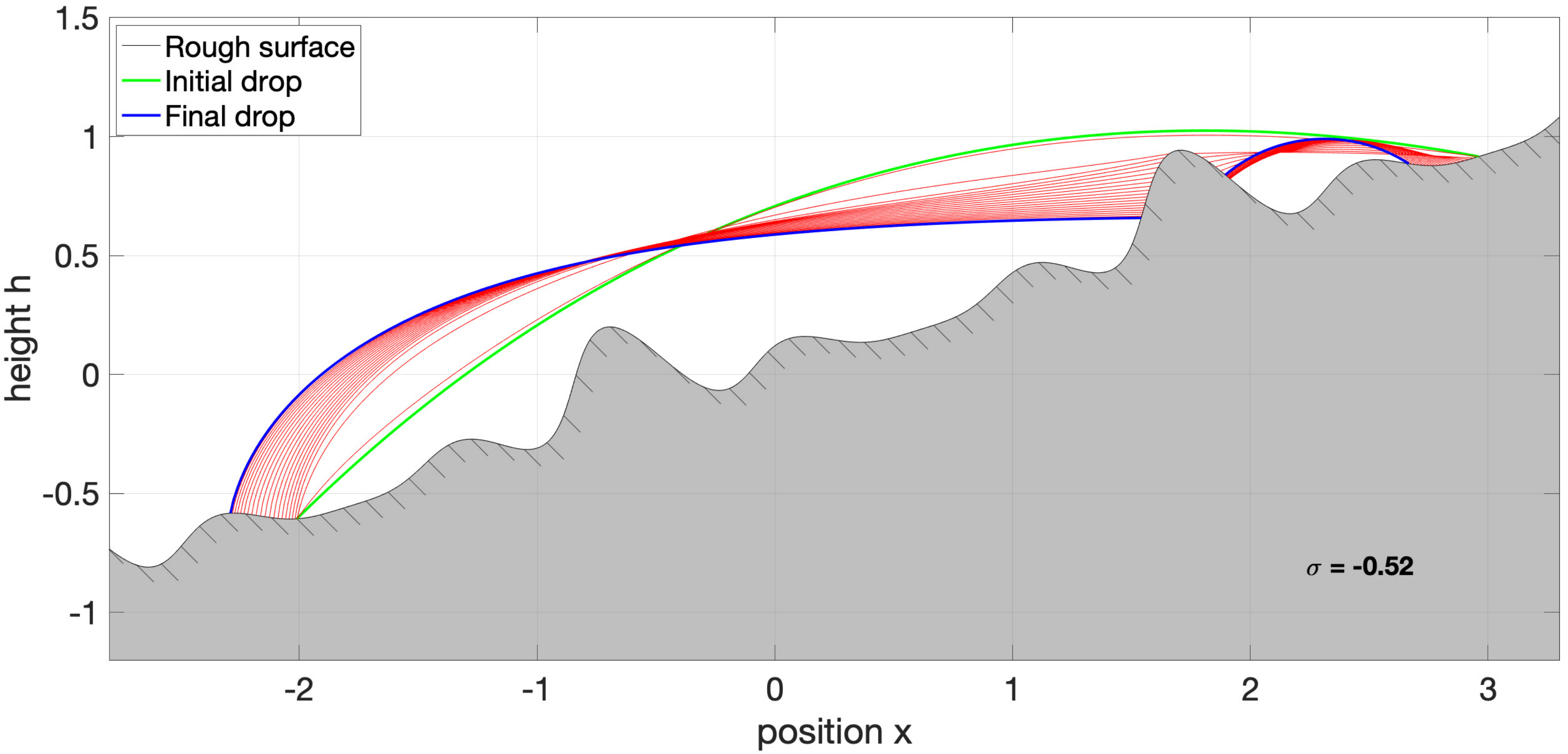} 
\caption{Evolution and splitting of one droplet into two droplets on an inclined groove-textured surface    using the projection scheme in Section \ref{sec_scheme}. Parameters: $\kappa=1, \, \beta=0.1$, number of moving grids  $N=200$, time step $\Delta t = 0.005$, final time $T=1$, splitting threshold $\eps=0.075$,  initial drop profile (green line) $h(x,t)$ in \eqref{h0_r1} with initial endpoints $a(0)=-2.1, b(0)=3.1$, Bond number $\Bo=0.5712.$,  relative adhesion coefficient $\sigma=-0.95$ and the inclined substrate with effective angle $\theta_0=0.3$ and \eqref{w1_r}. The evolution is shown using red lines at equal time intervals and the final profiles of two new droplets are shown in a blue line.  }\label{fig_r1r2}
\end{figure}

\subsubsection{Example 2: Two droplets merged together in the Utah teapot.}

We use the  Utah teapot, which is well-known in computer graphics history, as a typical inclined groove-textured substrate to demonstrate the merging of two droplets. The  Utah teapot can be constructed by several cubic 
B\'ezier curves \cite{bohm1984survey} connecting the following ten points $x_i, y_i$, $i=1, \cdots, 10$ as listed in Table \ref{table2}. 
\begin{table}
\caption{Ten points used in B\'ezier curve fitting of geometry of a Utah teapot}\label{table2}
\begin{tabular}{|c|c|c|c|c|c|c|c|c|c|c|}
\hline 
$i$ & 1 & 2 & 3 & 4 & 5 & 6 & 7 & 8 & 9 & 10 \\ 
\hline 
$x_i$ & -2 & $-\frac43$ & $-\frac23$ & 0 & $\frac23$ & $\frac43$ & 2 & $2.655$ & 2.846 & 4 \\ 
\hline 
$y_i$ & 0.78 & 0 & 0 & 0 & 0 & 0 & 0.78 & $1.142$ & 2.146 & 2.5 \\ 
\hline 
\end{tabular} 
\end{table}
For the bottom of the teapot, we use $(x_i, y_i)$ for $i=1, \cdots, 4$ and  $(x_i, y_i)$ for $i=4, \cdots, 7$. For the mouth of the teapot, we use $(x_i, y_i)$ for $i=7, \cdots, 10$. Assume the inclined groove-textured substrate is expressed by a parametric curve  $(x(\ell), y(\ell))$. Let $\ell(x)$ be the inverse function of $x(\ell)$, then  $w(x)=y(\ell(x))$ in \eqref{wet-eq-r}. 

Now we take the physical parameters as $\kappa=20, \, \beta=1$ and the relative adhesion coefficient as $\sigma=-0.78$. Assume the initial droplet $1$ is 
\begin{equation}\label{h01_tea}
h(x,0)=4.5(x-a(0))(b(0)-x)+w(a(0))+\frac{[w(b(0))-w(a(0))](x-a(0))}{b(0)-a(0)}
\end{equation}
 with initial endpoints $a(0)=1.9, b(0)=2.2$; as shown in Fig \ref{fig_teapot} with a magenta double-dotted line. 
Assume the initial droplet $2$ as 
\begin{equation}\label{h02_tea}
h(x,0)=7.8(x-a(0))(b(0)-x)+w(a(0))+\frac{[w(b(0))-w(a(0))](x-a(0))}{b(0)-a(0)}
\end{equation}
 with initial endpoints $a(0)=2.4, b(0)=2.9$ as shown in Fig \ref{fig_teapot} with a green double-dotted line.  The corresponding effective Bond number can be calculated according to Remark \ref{rem1} with effective inclined angle $\theta_0=0.226\pi$, $\Bo=0.0832$ for Droplet $1$ while $\Bo=0.7861$ for Droplet $2$.
 In the numeric scheme, we use $N=1000$ moving grids uniformly in $(a(t), b(t))$ and the merging threshold $\eps=0.01$. We take the same final time $T=12$ with time step $\Delta t = 0.05$.  
Without merging,
 the dynamics at equal time intervals of Droplet $1$ and Droplet $2$ are shown separately as comparisons in Fig \ref{fig_teapot} (upper/middle) with the final profile at $T=12$ using a solid magenta line for Droplet 1 and a solid green line for Droplet 2.
The small magenta Droplet 1 (upper) shows slow capillary rise, while the large green Droplet 2 (middle) moves down fast due to gravitational effect.
However, with the same parameters and same initial profiles (double-dotted lines), the dynamics at equal time intervals for the two droplets placed together in the Utah teapot are shown in Fig \ref{fig_teapot} (down). The two droplets will merge together at $T=3$ with the solid magenta/green lines for Droplet 1/Droplet 2  and then they continue to  move down as a new big droplet as shown in thin blue lines. The final profile of the new big droplet at $T=12$ is shown in a solid blue line.

 \begin{figure}
\includegraphics[height=6.08cm]{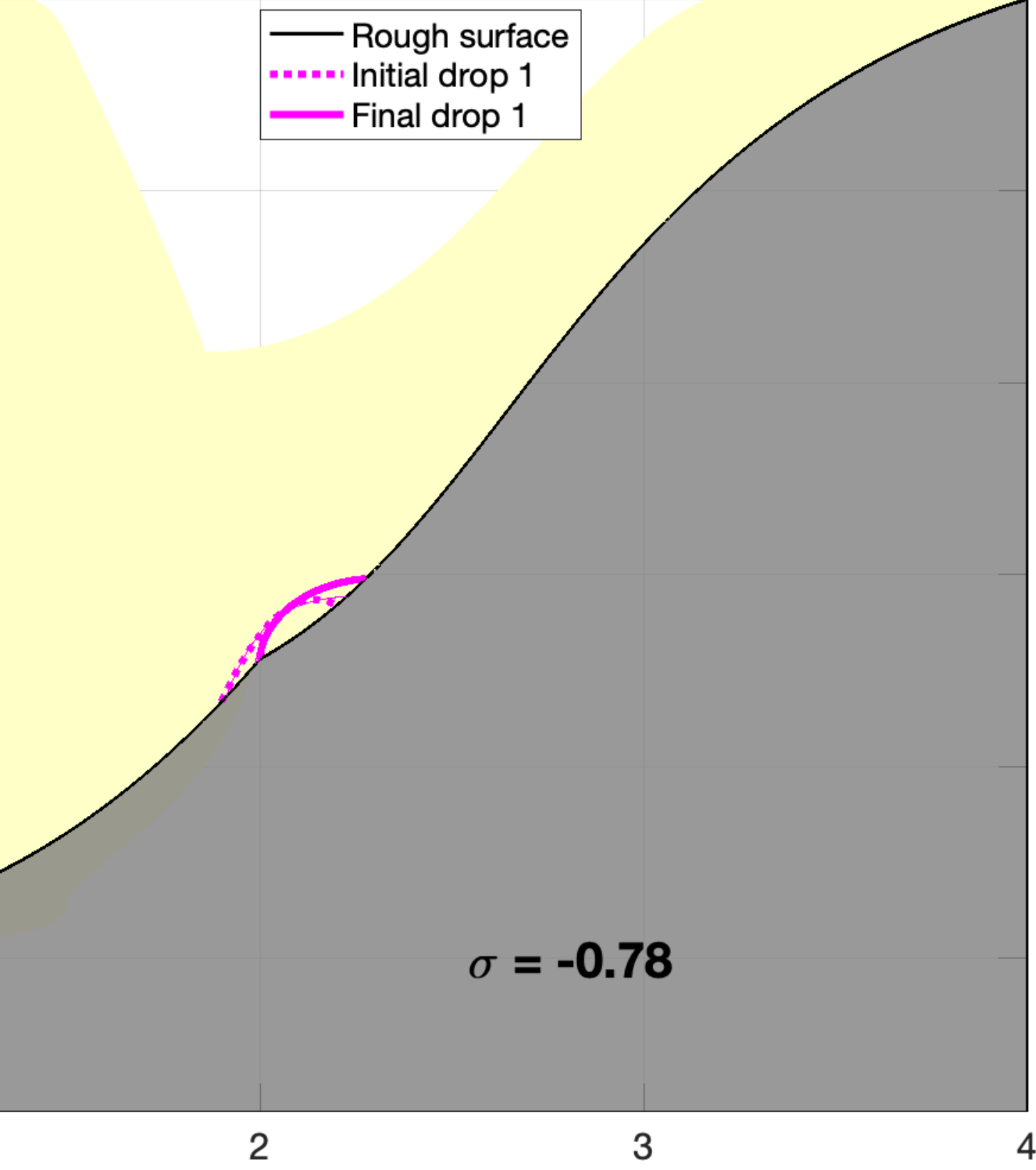} \includegraphics[height=6.08cm]{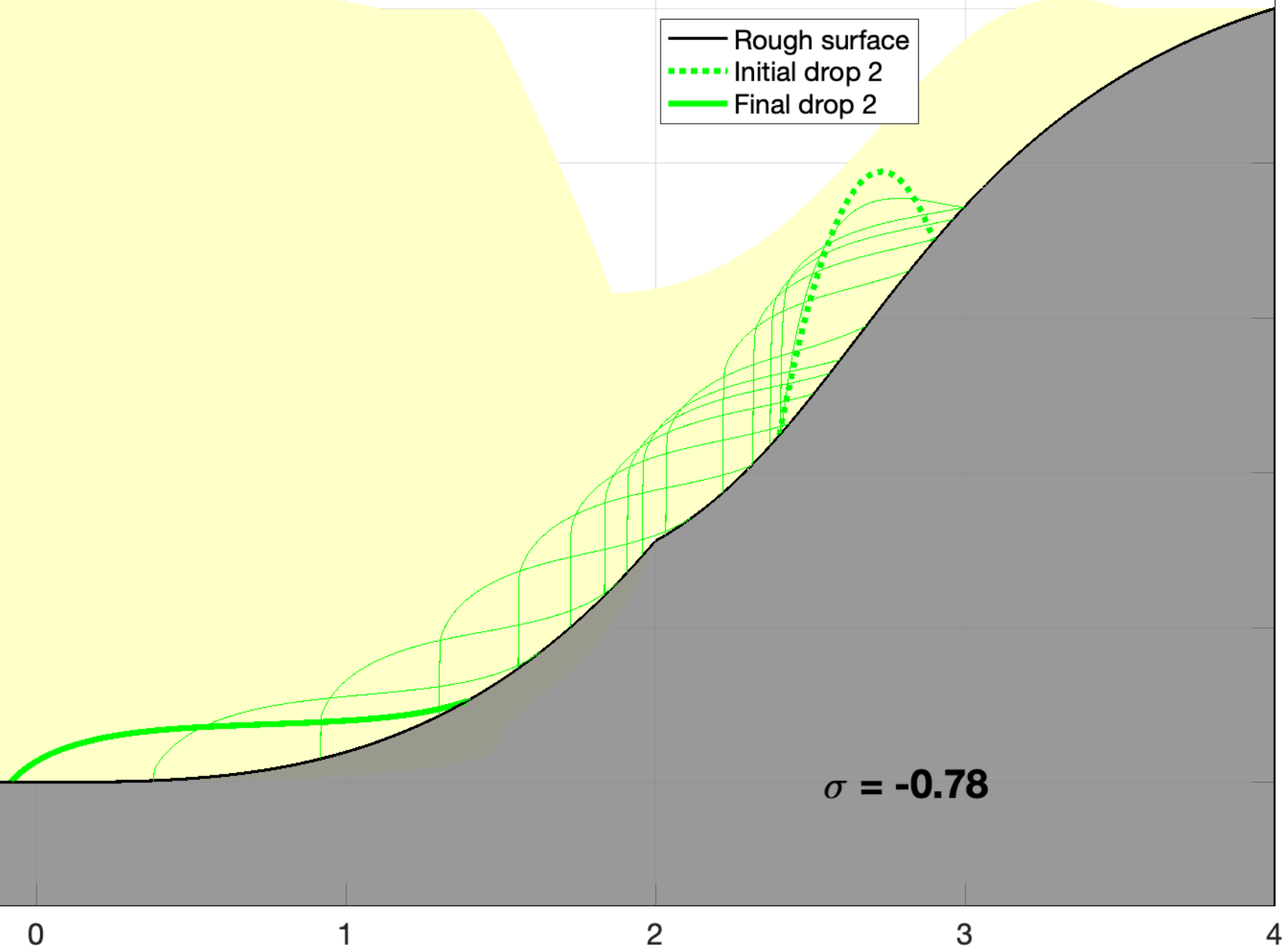} \includegraphics[height=7.0cm]{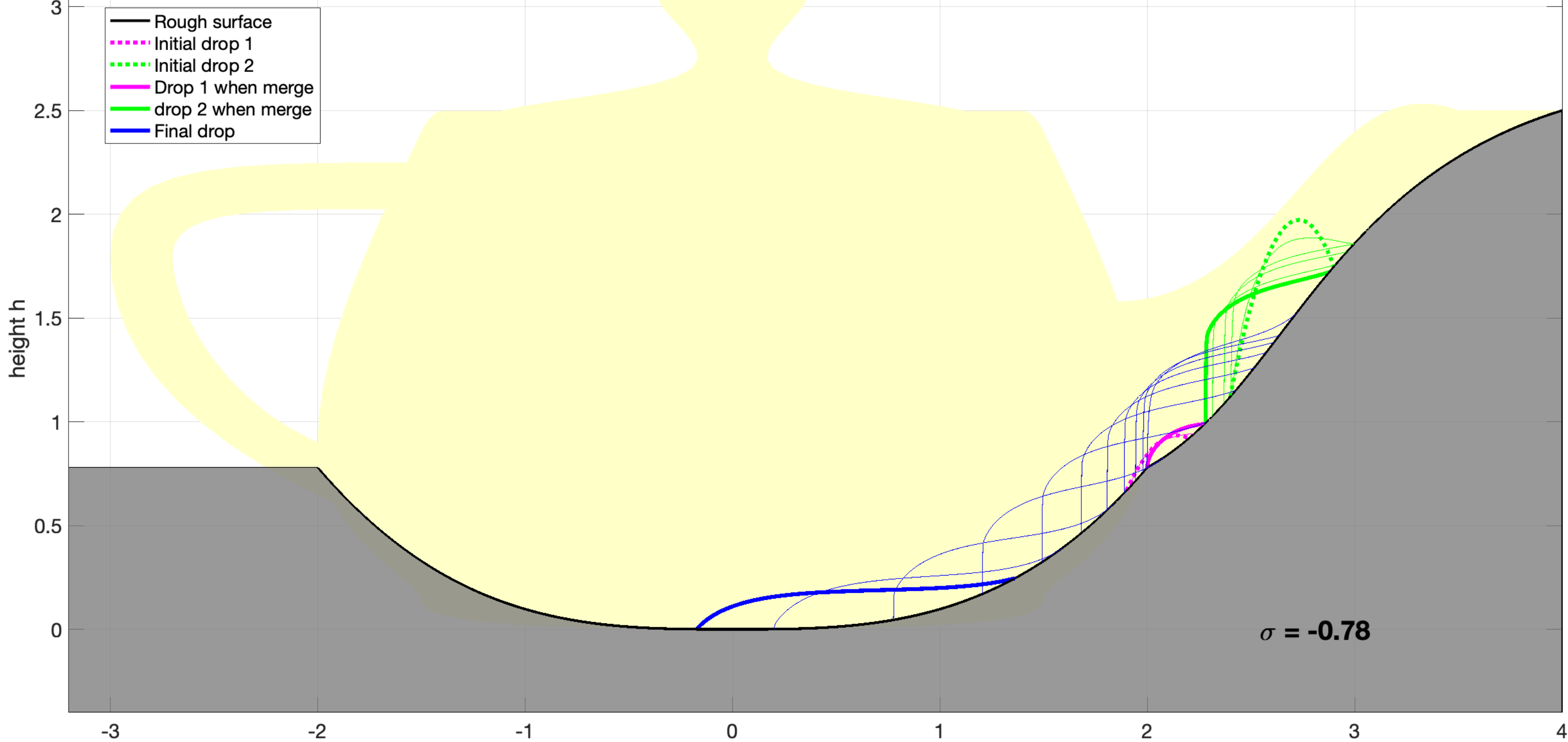} 
 \caption{Evolution of two partially wetting droplets in the Utah teapot at equal time intervals using the scheme in Section \ref{sec_scheme}. Parameters: number of moving grids in drop $N=1000$, time step $\Delta t = 0.05$, final time $T=12$, $\kappa=20, \, \beta=1$, relative adhesion coefficient $\sigma=-0.78$, merging threshold $\eps=0.01$, Bond number $\Bo=0.0832$ for Droplet $1$ and $\Bo=0.7861$ for Droplet $2$, initial Droplet $1$ profile (magenta double-dotted line) given in \eqref{h01_tea} with $a(0)=1.9, b(0)=2.2$ and initial Droplet $2$ profile (green double-dotted line) given in \eqref{h02_tea} with $a(0)=2.4, b(0)=2.9$.   (upper left) Droplet 1 with slow capillary rise; (upper right) Droplet 2 moves down fast due to gravitational effect; (down) Dynamics of two droplets: merge together and then  moves down as a new big droplet with final profile shown in solid blue lines.}\label{fig_teapot}
\end{figure}

\section*{Acknowledge}
The authors would like to thank Prof. Tom Witelski for some helpful suggestions. 
J.-G. Liu was supported in part by the National Science Foundation (NSF) under award DMS-1812573.

\bibliographystyle{abbrv}
\bibliography{dropevi}

\end{document}